\documentclass[12pt]{amsart}

\usepackage{amscd, amssymb, pgflibraryarrows, tikz,enumitem, hyperref, xcolor,textcomp,ragged2e,blindtext,multicol,trig,mathdots}
\usepackage[all,cmtip]{xy}
\usetikzlibrary{calc,3d,decorations.markings,shapes}
\usetikzlibrary{positioning,automata}
\tikzset{midarrow/.style = {postaction=decorate, decoration={markings,mark=at position .5 with \arrow{stealth}}}}
\newcommand{\Obj}{\operatorname{Obj}}

\usepackage{geometry}
\newtheorem{thm}{Theorem}[section]
\newtheorem{cor}[thm]{Corollary}
\newtheorem{lem}[thm]{Lemma}
\newtheorem{prop}[thm]{Proposition}

\theoremstyle{definition}
\newtheorem{dfn}[thm]{Definition}
\newtheorem{dfns}[thm]{Definitions}

\theoremstyle{remark}
\newtheorem{rmk}[thm]{Remark}
\newtheorem{rmks}[thm]{Remarks}
\newtheorem{example}[thm]{Example}
\newtheorem{examples}[thm]{Examples}

\newtheorem{nota}[thm]{Notation}

\numberwithin{equation}{section}

\newcommand{\ZZ}{\mathbb{Z}}
\newcommand{\NN}{\mathbb{N}}

\newcommand{\Gg}{\mathcal{G}}

\newcommand{\Cc}{\mathcal{C}}

\newcommand{\Sk}{\operatorname{Sk}}
\newcommand{\Aut}{\operatorname{Aut}}

\begin{document}

\title[Higher-rank trees]{Higher-rank trees arising from polyhedral graphs}

\begin{abstract}

We introduce a new family of higher-rank graphs, whose construction was inspired by the graphical techniques of Lambek \cite{Lambek} and Johnstone \cite{Johnstone} used for monoid and category embedding results. We show that they are planar rank-$k$ trees 
for $2 \le k \le 4$. We also show that higher-rank trees differ from rank-$1$ trees by giving examples of higher-rank trees having properties which are impossible for rank-$1$ trees. Finally, we collect more examples of higher-rank planar trees which are not in our family.
\end{abstract}

\date{\today}
\thanks{This research was supported by ARC Discovery Projects Grant 200100155}

\author{David Pask}
\address{David Pask \\ School of Mathematics and
Applied Statistics  \\
The University of Wollongong\\
NSW  2522\\
AUSTRALIA} \email{david.a.pask@gmail.com}

\maketitle
  
\section{Introduction}

Higher-rank graphs (or $k$-graphs) were first described by the author and Kumjian in \cite{KP2} as a graphical model for the higher-rank Cuntz-Krieger algebras of \cite{RobSteg1}. The $C^*$-algebra associated to a higher-rank graph has attracted wide attention from researchers in operator algebras (see \cite{E,FGLP,LV,Sp,Y1,Y} for example) as they are highly tractable and provide examples of classifiable $C^*$-algebras.
There is now a growing interest in higher-rank graphs from researchers in algebra (see \cite{CFaH,LPIS,Pre,Ros} amongst others). Lately there has been an interesting investigation of higher-rank graphs from a combinatorial/topological point of view (see \cite{bkps,KKQS,KPSW,PQR1,PQR2} for instance). From this work we now know that a higher-rank graph has a well-defined fundamental groupoid and a topological realisation which has the same fundamental group. The results in the most recent paper \cite{bkps}, has opened many subtle questions about the relationship between ordinary (rank-one) graphs and higher-rank graphs.

For example, the major question addressed in \cite{bkps} was the embedabilty of a higher-rank graph in its fundamental groupoid. This is always true for ordinary (rank-one) graphs but not for rank-two and higher. The embedability property has an impact on an important structure theorem for the $C^*$-algebra of a higher-rank graph: If a connected higher-rank graph $\Lambda$ embeds in its fundamental groupoid, then the $C^*$-algebra $C^* ( \Sigma )$ of its universal cover $\Sigma$ is type $I_0$. Then, under a further technical condition (which holds for rank-one)  $C^* ( \Sigma)$ is Morita equivalent to an abelian $C^*$-algebra (see \cite[Theorem 4.1]{bkps}).

As a higher-rank graph with a single vertex is a monoid, it was natural to look for solutions to the embedding problem in the study of monoids embedding in their enveloping group. The author happened upon the work of Ma{l}'cev in \cite{Malcev} and Lambek in \cite{Lambek}. Lambek comes up with a geometric condition called the \emph{polyhedral condition} (P), see \cite[p.\ 35]{Lambek}. Condition (P)  involves a system of equations derived from a polyhedral graph.
See \cite{Holl} for a good survey of this work. In \cite{Johnstone} Johnstone provides a similar construction when looking for  more general results for embedding a category in a groupoid. 
The authors of \cite{bkps} tried to refine the techniques of \cite{Lambek,Johnstone} to apply them to the embedabilty problem for higher-rank graphs in their fundamental groupoid but ran into technical difficulties, and this work was put aside.

This paper arose from solo work the author undertook after the project  \cite{bkps} was completed. The author took interest in the combinatorial object which parameterises the system of equations used in the work of \cite{Lambek,Johnstone}. In particular, the system of equations (as appears in Lambek's condition (P)) is reminiscent of the equations for bi-coloured commuting squares in a higher-rank graph. Many examples provided by \cite{Lambek, B} turn out to define higher-rank graphs of rank $2$ or $3$. Not only that, they have trivial fundamental group. 

Hence these examples are new instances of higher-rank trees.  The main purpose of the first part of this paper is to investigate their graphical properties, as there are relatively few examples of higher-rank trees available (other than standard constructions involving known rank-$1$ graphs). In the first two-thirds of this paper we produce a new class $\mathcal{L}_{\mathfrak{P}}$ of higher-rank trees from polyhedral graphs whose properties we can readily compute:

\noindent
\textbf{Theorem A:} To each non-degenerate polyhedral graph $\mathfrak{P}=(P,A)$ we construct a higher-rank tree $T \in \mathcal{L}_{\mathfrak{P}}$ which
\begin{enumerate}[leftmargin=0.5cm]
\item[$\bullet$] has a $1$-skeleton $\Sk_T$ which is is planar and  satisfies $| \Sk_T^1 | - 2 | \Sk_T^0 | +4 = 0$;
\item[$\bullet$] is connected, singly connected, locally convex and acyclic, that is $H_i (T)=0$ for $i \ge 1$;
\item[$\bullet$] embeds in its fundamental groupoid and up to quasi-isomorphism, has rank $2,3,4$.
\end{enumerate}

A long-term reason for wanting to find out more about such objects lies in the study of totally disconnected locally compact (tdlc)  simple groups, where automorphism groups of (one-dimensional, homogeneous) trees play an important role (see \cite[\S 3]{W}). This paper therefore is part of an ambition to produce new examples of simple tdlc. We begin to explore this  possiblility, with a potential example, in section~\ref{sec:tailpiece}.

In this manuscript we begin in section~\ref{sec:one} by accumulating some background material and notation  for higher-rank graphs, in particular their description in terms of coloured directed graphs with a collection of bi-coloured commuting squares describing their defining relations.

In section~\ref{sec:two} we begin the construction of the combinatorial object used in \cite{Lambek,Johnstone}. In particular, for each polyhedral graph appearing in their methods, they parameterise a system of equations indexed by half-lines. Johnstone then constructs a directed graph $E$, called a quadrangle club, so-called  as it is built out of directed squares. The system of equations (the Lambek half-arc equations) in question is then parameterised by subgraphs with four vertices and edges, called quadrangles.  In section~\ref{sec:three} we show that this system of equations could be coloured in the same way that the bi-coloured commuting squares in the 1-skeleton associated to a higher-rank graph are coloured. This computation is facilitated by the fact  quadrangle club graph consists of bi-coloured square subgraphs and so is bipartite. Furthermore, this set of bi-coloured commuting squares $\mathcal{C}$ coming from a convex polyhedron is complete, and so gives rise to a higher-rank graph $\Lambda_{E,\mathcal{C}}$, see Theorem~\ref{thm:polygiveshrg}. Since the convex polyhedron is planar, this means that the dimension of resulting the higher-rank graph lies between two and four by the celebrated \emph{four-color theorem} (see \cite{AH1,AHK,AH2}). Furthermore, some basic graph theory allows us to characterise which polyhedral graphs give rise to rank-$2$ graphs (see Proposition~\ref{prop:2-colour}).

The results in \cite{KKQS} show that higher-rank graphs have a topological realisation which is consistent with their fundamental group, as defined in \cite{PQR1}. Since the higher-rank graphs we have constructed begin their life as a graph embedded in a sphere, it would then be natural to suppose that the fundamental group of these new higher-rank graphs is trivial. We conjecture that, under a certain non-degeneracy condition, all planar higher-rank trees we have constructed have dimension less than or equal to four.

In section~\ref{sec:five} we address the question of computing the fundamental group of $\Lambda_{E,\mathcal{C}}$. To compute the fundamental group of a higher-rank graphs, we use the techniques of \cite{KPW} which begins with the computation of a maximal spanning tree of the underlying $1$-skeleton $E$. In section~\ref{sec:greedy} we produce an algorithm for doing this which is weighted towards edges appearing on the left of the bi-coloured commuting squares (for demonstrative reasons only, see Examples~\ref{ex:treex}). We then apply \cite[Theorem 5.1]{KPW} to produce a presentation of the fundamental group and argue that it is trivial by systematically showing all the generators of the group are trivial (see Theorem~\ref{thm:fgistriv}). This computation is facilitated by the fact that all paths in the quadrangle club are at most length two. Finally, we use the results of \cite{bkps}, to show in Corollary~\ref{cor:pthrg-embedd} that these new planar higher-rank graphs embed in their fundamental groupoid.


Finally, moving away from the class $\mathcal{L}_{\mathfrak{P}}$, we provide what the author believes  are three good reasons for
studying finite higher-rank trees from a purely graphical perspective. These reasons are interesting because they are examples where the higher dimensional behaviour varies from the familiar one dimensional behaviour. 

\textbf{First}, a rank-$1$ tree is planar, however is is possible to find (in dimension greater than three) a higher rank tree which is not planar. 

\textbf{Second}, it is well-known that a rank-$1$ tree does not admit a (fixed-point) free automorphism of order three. However, for a higher-rank 
tree of higher dimension, the underlying category is not a free category, so there can be cycles (albeit trivial in the fundamental group). We give an example of a rank-$2$ planar tree which admits a (fixed-point) free automorphism of order $3$.

\textbf{Third}, it is well-known that a rank-$1$ tree with an odd number of vertices always has a (fixed-point free) automorphism. We give an example of a planar higher-rank tree with an odd number of vertices which has no automorphisms.

We finish by giving classes of examples of higher-rank trees which are already in the literature.
The most important of which are the $\Omega_{k,m}$ graphs whose properties we summarise:
\par\noindent\textbf{Theorem B:} To each integer $k \ge 1$ and number $0 \neq m \in \NN^k$ there is a higher rank tree  $\Omega_{k,m}$ which
\begin{enumerate}[leftmargin=0.5cm]
\item[$\bullet$] is planar if $k=2$, or $k=3$ and $m=(1,1,1)$;
\item[$\bullet$] does not belong to $\mathcal{L}_{\mathfrak{P}}$;
\item[$\bullet$] is connected, singly connected, locally convex and acyclic, that is $H_i ( \Omega_{k,m} ) =0$ for $i \ge 1$;
\item[$\bullet$] has rank $k$ and embeds in its fundamental groupoid.
\end{enumerate}

 

The author would like to thank the other authors of \cite{bkps} for stimulating conversations about the embedding of higher-rank graphs in their fundamental groupoid, in particular the examples which stimulated his interest in studying them more deeply in this manuscript.

\section{Background} \label{sec:one}

Let $k \ge 1$ and $\NN^k$ denote the free abelian monoid with generators
$\{ \varepsilon_i : 1 \le i \le k \}$. Let $\le$ be the usual coordinatewise order on $\NN^k$. 
For $m,n\in \NN^k$, we write $m \vee n $ for their coordinate-wise maximum and $m \wedge n$ for their coordinatewise minimum.
We write $\mathbf{1}_k$ or just $\mathbf{1}$ for $(1, \ldots , 1) \in \NN^k$. A \emph{directed graph} is a quadruple $E=(E^0,E^1,r,s)$ consisting of finitely many vertices $E^0$ and edges $E^1$ whose direction is given by the maps $r,s : E^1 \to E^0$. The sets $E^n$, $n \ge 0$ denote the collection of paths of length $n$ in $E$. The set of finite paths, $E^* = \cup_{n \ge 0} E^n$ forms a category with the roles of the range and source maps interchanged.

\subsection{Coloured graphs} \label{sec:CG}
We briefly review the definitions given in \cite[\S 3]{hrsw} which themselves have been tailored for connections with higher-rank graphs: For $k \ge 1$, a \emph{$k$-coloured graph} $E=(E^0,E^1,r,s)$ is a directed graph along with a \emph{colour map} $c_E : E^1 \to \{c_1,\dots,c_k\}$. By considering $\{c_1,\dots , c_k\}$ as generators of the free group $\mathbb{F}_k$, we extend $c_E:E^* \setminus E^0 \to \mathbb{F}_k^+$ by $c_E (\mu_1 \cdots \mu_n)=c_E (\mu_1)c_E (\mu_2) \cdots c_E (\mu_n)$. We will drop the subscript from $c_E$ if there is no risk of confusion. For $k$-coloured directed graphs $E$ and $F$, a \emph{coloured-graph morphism} $\phi: F \to E$ is a graph morphism satisfying $c_E \circ \phi^1 = c_F$. 

\vspace{1em}
\noindent\begin{minipage}[c]{0.78\linewidth}
\begin{example}\label{ex:model coloured graphs} \label{dfn:Ekm}
For $k \ge 1$ and $m\in \mathbb{N}^k$, the $k$-coloured graph $(E_{k,m},c_E)$ is defined by 
$E_{k,m}^0=\{n\in\mathbb{N}^k:0\le n\le m\}$, $E_{k,m}^1=\{f_i^n : n,n+\varepsilon_i\in E_{k,m}^0\}$, with $r(f_i^n)=n$, $s(f_i^n)=n+\varepsilon_i$ and $c_{E}(f_i^n)=c_i$. 
The 2-coloured graph $(E_{2,(\varepsilon_1+\varepsilon_2)},c_E)$ is used often, and is depicted to the right. For $2$-coloured graphs, the convention is to draw edges with colour $c_1$ blue (or solid) and edges with colour $c_2$ red (or dashed).
\end{example}
\end{minipage}
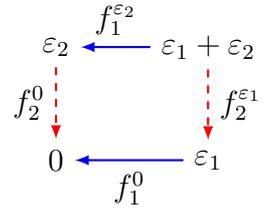
\begin{minipage}[c]{0.2\linewidth}
\[
\begin{tikzpicture}[>=stealth,yscale=1.5,xscale=2]
\node (00) at (0,0) {$0$};
\node (10) at (1,0) {$\varepsilon_1$}
	edge[-latex,thick,blue] node[auto,black] {$f_1^{0}$} (00);
\node (01) at (0,1) {$\varepsilon_2$}
	edge[-latex,thick,red,dashed] node[auto,black,swap] {$f_{2}^{0}$} (00);
\node (11) at (1,1) {$\varepsilon_1+\varepsilon_2$}
	edge[-latex,thick,blue] node[auto,black,swap] {$f^{\varepsilon_2}_{1}$} (01)
	edge[-latex,thick,red,dashed] node[auto,black] {$f_{2}^{\varepsilon_1}$} (10);
\end{tikzpicture}
\]
\end{minipage}

\begin{dfn} \label{dfn:square}
Let $(E,c_E')$ be a $k$-coloured graph and $i \neq j \leq k$. A coloured graph morphism $\phi: (E_{k,(\varepsilon_i+\varepsilon_j)},c_E) \to (E,c_E')$ is a \emph{square} in $E$.
\end{dfn}

\noindent
One may represent a square $\phi$ as a labelled version of $(E_{2,(\varepsilon_i+\varepsilon_j)},c_E)$. For instance the $2$-coloured graph below on the left has only one square $\phi$, shown to its right, given by $\phi(n) = v$ for all $n \in E_{2,{(\varepsilon_1+\varepsilon_2)}}^0$, $\phi( f_1^0) = \phi( f_1^{\varepsilon_2}) = e$ and $\phi( f_2^0 ) = \phi( f_2^{\varepsilon_1} ) = f$. 
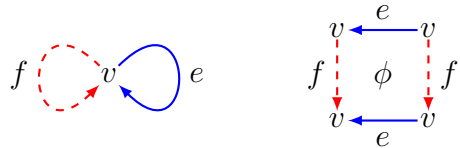
\[
\begin{tikzpicture}[>=stealth,scale=1.2]
\node[inner sep=0.8pt] (v) at (0,0) {$v$}
	edge[loop,blue,thick,latex-,in=45,out=-45,looseness=15] node[auto,black,swap] {$e$} (v)
	edge[loop,red,thick,dashed,latex-,in=135,out=-135,looseness=15] node[auto,black] {$f$} (v);
\begin{scope}[xshift=2.5cm,yshift=-0.5cm]
\node at (0.5,0.5) {$\phi$};
\node[inner sep=0.8pt] (00) at (0,0) {$v$};
\node[inner sep=0.8pt] (10) at (1,0) {$v$}
	edge[-latex,blue,thick] node[auto,black] {$e$} (00);
\node[inner sep=0.8pt] (01) at (0,1) {$v$}
	edge[-latex,red,dashed,thick] node[auto,black,swap] {$f$} (00);
\node[inner sep=0.8pt] (11) at (1,1) {$v$}
	edge[-latex,blue,thick] node[auto,black,swap] {$e$} (01)
	edge[-latex,red,dashed,thick] node[auto,black] {$f$} (10);
\end{scope}
\end{tikzpicture}
\]

\noindent
Let $\Cc_E = \{ \phi : (E_{k,\varepsilon_i+\varepsilon_j} ,c_E)\to (E,c_E') : 1 \le i \neq j \le k  \}$ denote the collection of squares in a $k$-coloured graph $(E,c_E)$.

\begin{dfns} \label{dfn:complete-squares}
A collection of squares $\Cc$ in $(E,c_E)$ is called \emph{complete} if for each $i \neq j \leq k$ and $c_ic_j$-coloured path $fg \in E^2$, there exists a unique $\phi \in \Cc$ such that $\phi(f_i^{0}) = f$ and $\phi(f_j^{\varepsilon_i}) = g$.

The uniqueness of $\phi$ gives a unique $c_jc_i$-coloured path $g'f'$ with $g' = \phi(f_j^{0})$ and $f' = \phi(f_i^{\varepsilon_j})$. We will write $fg \sim_\Cc g'f'$ and refer to elements $(fg,g'f')$ of this relation (on $\{ (\alpha, \beta ) \in E^2 \times E^2 : s(\alpha)=s(\beta),r(\alpha)=r(\beta), c(\alpha)=c_ic_j, c(\beta)=c_jc_i , i \neq j \}$) as \emph{commuting squares}. 
\end{dfns}

\begin{rmk}
According to \cite[Theorem 4.4, Theorem 4.5]{hrsw} If a coloured graph $(E,c)$
has a complete set of squares and is such that there are no paths of length three (i.e. $E^3=\emptyset$). Then a condition, called \emph{associativity}, used for squares in $k$-coloured graphs $(E,c_E)$, where $k \ge 3$ does not come into play (see \cite[Equation (3.2)]{hrsw}).
\end{rmk}

\subsection{Higher-rank graphs}

A \emph{higher-rank graph} (or a \emph{rank-$k$ graph}) is a countable category $\Lambda$ with a \emph{degree} functor $d:\Lambda \to \NN^k$, $k \ge 1$, satisfying the \emph{factorisation property}: if $\lambda\in\Lambda$ and $m,n\in\NN^k$ are such that $d(\lambda)=m+n$, then there are unique $\mu,\nu \in \Lambda$ with $d(\mu)=m$, $d(\nu)=n$ and $\lambda=\mu\nu$. 

Given $m \in \NN^k$ we define $\Lambda^m := d^{-1}(m)$. Following \cite{PQR1}, for $v,w \in \Lambda^0$ and $F \subseteq \Lambda$ define $vF := r^{-1}(v) \cap F$, $Fw := s^{-1}(w) \cap F$, and $vFw := vF \cap Fw$. If $V \subset \Lambda^0$, then $V\Lambda = r^{-1} (V)$ and $\Lambda V = s^{-1} (V)$.

The factorisation property allows us to identify $\Lambda^0$ with $\operatorname{Obj}(\Lambda)$, and we call its elements \emph{vertices}. A vertex $v \in \Lambda^0$ is a \emph{sink} if $s^{-1} (v) = \emptyset$, a vertex $w \in \Lambda^0$ is a \emph{source} if $r^{-1} (w) = \emptyset$. By the factorisation property, for each $\lambda\in\Lambda$ and $m \leq n \leq d(\lambda)$, we may write $\lambda=\lambda'\lambda''\lambda'''$, where $d(\lambda')=m, d(\lambda'') = n-m$ and $d(\lambda'')=d(\lambda)-n$; then $\lambda (m,n) :=\lambda''$. For more information about higher-rank graphs see \cite{KP1,RSY1} for example.

We define the \emph{$1$-skeleton} of a $k$-graph $\Lambda$ to be the $k$-coloured graph $(\Sk_\Lambda,c)$ given by $\Sk_\Lambda^0 = \Obj(\Lambda)$, $\Sk_\Lambda^1 = \bigcup_{i\leq k}\Lambda^{\varepsilon_i}$, with range and source as in $\Lambda$. The colouring map $c : \Sk_\Lambda^1 \to \{c_1,\dots,c_k\}$ is given by $c(f) = c_i$ if and only if $f \in \Lambda^{\varepsilon_i}$. The $1$-skeleton $(\Sk_\Lambda,c)$ comes with a canonical set of squares $\Cc_\Lambda := \{\phi_\lambda : \lambda \in \Lambda^{\varepsilon_i+\varepsilon_j}: i \neq j \leq k  \}$, where $\phi_\lambda : E_{k,\varepsilon_i+\varepsilon_j} \to \Sk_\Lambda$ is given by $\phi_\lambda(\varepsilon_\ell^{n}) = \lambda(n,n+\varepsilon_\ell)$ for each $n \leq \varepsilon_i+\varepsilon_j$ and $\ell = i,j$. The collection $\Cc_\Lambda$ is complete by \cite[Lemma 4.2]{hrsw}.

A \emph{quasi-morphism} from a rank-$k$ graph $( \Omega , d_\Omega )$ to a rank-$\ell$ graph $( \Lambda , d_\Lambda )$ is a pair $( \phi , f)$ consisting of a functor $\phi : \Omega \to \Lambda$ and a homomorphism $f: \NN^k \to \NN^\ell$ such that $d_\Lambda \circ \phi = f \circ d_\Omega$. 

\begin{examples} \label{ex:omega}
\begin{enumerate}[leftmargin=1cm,label={(\roman *)}]
\item The path category $E^*$ of a directed graph is a rank-$1$ graph. Indeed, every rank-$1$ graph occurs in this way.
\item Adapting \cite{RSY1} slightly, for each $m \in \NN^k$, $m \neq 0$ and $k \ge 1$, we define a rank-$k$ graph $\Omega_{k,m}$ as follows:
Let $\Omega_{k,m} = \{(p,q) \in \NN^k \times \NN^k : p \le q \le m\}$. 
We may define a range map $r(p,q) = (p,p)$ and a source map $s(p,q) = (q,q)$, and composition $( \ell,n) = ( \ell,m) (m,n)$. We define the degree map on $\Omega_{k,m}$ by $d(p,q) = q-p$. We may identify $\Omega_{k,m}^0$ with $\{p \in \NN^k : p \le m\}$ via the map $(p,p) \mapsto p$. 
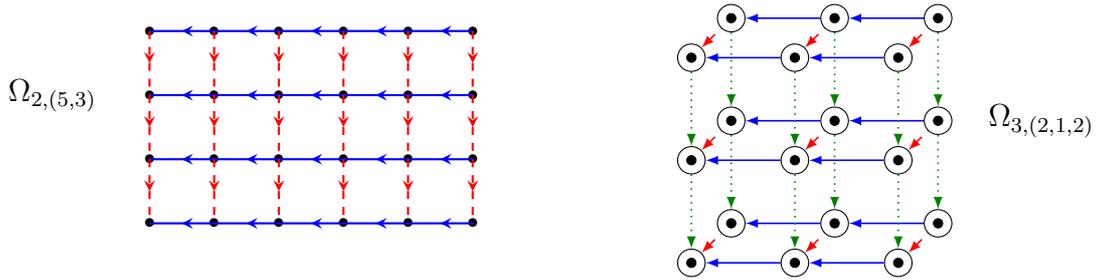
\[
\begin{tikzpicture}[scale=0.85]

\node[inner sep=2.5pt] at (0,0) {\tiny $\bullet$};
\node[inner sep=2.5pt] at (0,1) {\tiny $\bullet$};
\node[inner sep=2.5pt] at (0,2) {\tiny $\bullet$};
\node[inner sep=2.5pt] at (0,3) {\tiny $\bullet$};


\node[inner sep=2.5pt] at (1,0) {\tiny $\bullet$};
\node[inner sep=2.5pt] at (1,1) {\tiny $\bullet$};
\node[inner sep=2.5pt] at (1,2) {\tiny $\bullet$};
\node[inner sep=2.5pt] at (1,3) {\tiny $\bullet$};


\node[inner sep=2.5pt] at (2,0) {\tiny $\bullet$};
\node[inner sep=2.5pt] at (2,1) {\tiny $\bullet$};
\node[inner sep=2.5pt] at (2,2) {\tiny $\bullet$};
\node[inner sep=2.5pt] at (2,3) {\tiny $\bullet$};


\node[inner sep=2.5pt] at (3,0) {\tiny $\bullet$};
\node[inner sep=2.5pt] at (3,1) {\tiny $\bullet$};
\node[inner sep=2.5pt] at (3,2) {\tiny $\bullet$};
\node[inner sep=2.5pt] at (3,3) {\tiny $\bullet$};


\node[inner sep=2.5pt] at (4,0) {\tiny $\bullet$};
\node[inner sep=2.5pt] at (4,1) {\tiny $\bullet$};
\node[inner sep=2.5pt] at (4,2) {\tiny $\bullet$};
\node[inner sep=2.5pt] at (4,3) {\tiny $\bullet$};


\node[inner sep=2.5pt] at (5,0) {\tiny $\bullet$};
\node[inner sep=2.5pt] at (5,1) {\tiny $\bullet$};
\node[inner sep=2.5pt] at (5,2) {\tiny $\bullet$};
\node[inner sep=2.5pt] at (5,3) {\tiny $\bullet$};

\draw[blue, thick,midarrow] (1,0)--(0,0);
\draw[blue, thick,midarrow] (2,0)--(1,0);
\draw[blue, thick,midarrow] (3,0)--(2,0);
\draw[blue, thick,midarrow] (4,0)--(3,0);
\draw[blue,thick,midarrow] (5,0)--(4,0);

\draw[blue, thick,midarrow] (1,1)--(0,1);
\draw[blue, thick,midarrow] (2,1)--(1,1);
\draw[blue, thick,midarrow] (3,1)--(2,1);
\draw[blue, thick,midarrow] (4,1)--(3,1);
\draw[blue,thick,midarrow] (5,1)--(4,1);

\draw[blue, thick,midarrow] (1,2)--(0,2);
\draw[blue, thick,midarrow] (2,2)--(1,2);
\draw[blue, thick,midarrow] (3,2)--(2,2);
\draw[blue, thick,midarrow] (4,2)--(3,2);
\draw[blue,thick,midarrow] (5,2)--(4,2);

\draw[blue, thick,midarrow] (1,3)--(0,3);
\draw[blue, thick,midarrow] (2,3)--(1,3);
\draw[blue, thick,midarrow] (3,3)--(2,3);
\draw[blue, thick,midarrow] (4,3)--(3,3);
\draw[blue,thick,midarrow] (5,3)--(4,3);

\draw[red, dashed, thick,midarrow] (0,1)--(0,0);
\draw[red, dashed, thick,midarrow] (0,2)--(0,1);
\draw[red, dashed, thick,midarrow] (0,3)--(0,2);

\draw[red, dashed, thick,midarrow] (1,1)--(1,0);
\draw[red, dashed, thick,midarrow] (1,2)--(1,1);
\draw[red, dashed, thick,midarrow] (1,3)--(1,2);

\draw[red, dashed, thick,midarrow] (2,1)--(2,0);
\draw[red, dashed, thick,midarrow] (2,2)--(2,1);
\draw[red, dashed, thick,midarrow] (2,3)--(2,2);

\draw[red, dashed, thick,midarrow] (3,1)--(3,0);
\draw[red, dashed, thick,midarrow] (3,2)--(3,1);
\draw[red, dashed, thick,midarrow] (3,3)--(3,2);

\draw[red, dashed, thick,midarrow] (4,1)--(4,0);
\draw[red, dashed, thick,midarrow] (4,2)--(4,1);
\draw[red, dashed, thick,midarrow] (4,3)--(4,2);

\draw[red, dashed, thick,midarrow] (5,1)--(5,0);
\draw[red, dashed, thick,midarrow] (5,2)--(5,1);
\draw[red, dashed, thick,midarrow] (5,3)--(5,2);

\node at (-1.5,2) {$\Omega_{2,(5,3)}$};

\begin{scope}[xshift=9cm,scale=1.6]

 \node[inner sep=2.5pt, circle, opacity=0] (vert000) at (0,0,0) [draw] {.};
\node[inner sep=2.5pt, circle, opacity=0] (vert100) at (1,0,0) [draw] {.};
 \node[inner sep=2.5pt, circle, opacity=0] (vert200) at (2,0,0) [draw] {.};
        
\node[inner sep=2.5pt, circle, opacity=0] (vert010) at (0,1,0) [draw] {.};
\node[inner sep=2.5pt, circle, opacity=0] (vert110) at (1,1,0) [draw] {.};
\node[inner sep=2.5pt, circle, opacity=0] (vert210) at (2,1,0) [draw] {.};
        
\node[inner sep=2.5pt, circle, opacity=0] (vert020) at (0,2,0) [draw] {.};
\node[inner sep=2.5pt, circle, opacity=0] (vert120) at (1,2,0) [draw] {.};
\node[inner sep=2.5pt, circle, opacity=0] (vert220) at (2,2,0) [draw] {.};
        
\node[inner sep=2.5pt, circle, opacity=0] (vert001) at (0,0,1) [draw] {.};
\node[inner sep=2.5pt, circle, opacity=0] (vert101) at (1,0,1) [draw] {.};
\node[inner sep=2.5pt, circle, opacity=0] (vert201) at (2,0,1) [draw] {.};
       
\node[inner sep=2.5pt, circle, opacity=0] (vert011) at (0,1,1) [draw] {.};
\node[inner sep=2.5pt, circle, opacity=0] (vert111) at (1,1,1) [draw] {.};
\node[inner sep=2.5pt, circle, opacity=0] (vert211) at (2,1,1) [draw] {.};
      
\node[inner sep=2.5pt, circle, opacity=0] (vert021) at (0,2,1) [draw] {.};
\node[inner sep=2.5pt, circle, opacity=0] (vert121) at (1,2,1) [draw] {.};
\node[inner sep=2.5pt, circle, opacity=0] (vert221) at (2,2,1) [draw] {.};

\node[inner sep=0pt, circle, fill=black] at (vert000) [draw] {.};

\node[inner sep=2.5pt, circle, anchor = north west] at (vert000.south east) {$ $};

\node[inner sep=0pt, circle, fill=black] at (vert001) [draw] {.};

\node[inner sep=2.5pt, circle, anchor = north west] at (vert001.south east) {$ $};

\node[inner sep=0pt, circle, fill=black] at (vert010) [draw] {.};

\node[inner sep=2.5pt, circle, anchor = north west] at (vert010.south east) {$ $};

\node[inner sep=0pt, circle, fill=black] at (vert011)  [draw] {.};

\node[inner sep=2.5pt, circle, anchor = north west] at (vert011.south east) {$ $};

\node[inner sep=0pt, circle, fill=black] at (vert020) [draw] {.};

\node[inner sep=2.5pt, circle, anchor = north west] at (vert020.south east) {$ $};

\node[inner sep=0pt, circle, fill=black] at (vert021) [draw] {.};

\node[inner sep=2.5pt, circle, anchor = north west] at (vert021.south east) {$ $};

\node[inner sep=0pt, circle, fill=black] at (vert100) [draw] {.};

\node[inner sep=2.5pt, circle, anchor = north west] at (vert100.south east) {$ $};

\node[inner sep=0pt, circle, fill=black] at (vert101) [draw] {.};

\node[inner sep=2.5pt, circle, anchor = north west] at (vert101.south east) {$ $};

\node[inner sep=0pt, circle, fill=black] at (vert110) [draw] {.};

\node[inner sep=2.5pt, circle, anchor = north west] at (vert110.south east) {$ $};

\node[inner sep=0pt, circle, fill=black] at (vert111) [draw] {.};

\node[inner sep=2.5pt, circle, anchor = north west] at (vert111.south east) {$ $};

\node[inner sep=0pt, circle, fill=black] at (vert120) [draw] {.};

\node[inner sep=2.5pt, circle, anchor = north west] at (vert120.south east) {$ $};

\node[inner sep=0pt, circle, fill=black] at (vert121) [draw] {.};

\node[inner sep=2.5pt, circle, anchor = north west] at (vert121.south east) {$ $};

\node[inner sep=0pt, circle, fill=black] at (vert200) [draw] {.};

\node[inner sep=2.5pt, circle, anchor = north west] at (vert200.south east) {$ $};

\node[inner sep=0pt, circle, fill=black] at (vert201) [draw] {.};

\node[inner sep=2.5pt, circle, anchor = north west] at (vert201.south east) {$ $};

\node[inner sep=0pt, circle, fill=black] at (vert210) [draw] {.};

\node[inner sep=2.5pt, circle, anchor = north west] at (vert210.south east) {$ $};

\node[inner sep=0pt, circle, fill=black] at (vert211) [draw] {.};

\node[inner sep=2.5pt, circle, anchor = north west] at (vert211.south east) {$ $};

\node[inner sep=0pt, circle, fill=black] at (vert220) [draw] {.};
\node[inner sep=2.5pt, circle, anchor = north west] at (vert220.south east) {$ $};

\node[inner sep=0pt, circle, fill=black] at (vert221) [draw] {.};
\node[inner sep=2.5pt, circle, anchor = north west] at (vert221.south east) {$ $};

\draw[style=semithick, blue,-latex] (vert100.west)--(vert000.east);
        
\draw[style=semithick,blue,-latex] (vert110.west)--(vert010.east);
        
\draw[style=semithick, blue,-latex] (vert120.west)--(vert020.east);
        
\draw[style=semithick, blue,-latex] (vert101.west)--(vert001.east);
       
\draw[style=semithick, blue,-latex] (vert200.west)--(vert100.east);
        
\draw[style=semithick, blue,-latex] (vert210.west)--(vert110.east);
        
\draw[style=semithick,  blue,-latex] (vert220.west)--(vert120.east);
        
\draw[style=semithick,  blue,-latex] (vert201.west)--(vert101.east);
        
\draw[style=semithick, dotted,green!50!black,-latex] (vert010.south)--(vert000.north);
        
\draw[style=semithick,dotted,green!50!black,-latex] (vert011.south)--(vert001.north);
        
\draw[style=semithick, dotted,green!50!black,-latex] (vert110.south)--(vert100.north);
        
\draw[style=semithick, dotted,green!50!black,-latex] (vert111.south)--(vert101.north);
        
\draw[style=semithick, dotted,green!50!black,-latex] (vert210.south)--(vert200.north);
        
\draw[style=semithick, dotted,green!50!black,-latex] (vert211.south)--(vert201.north);
        
\draw[style=semithick, dotted,green!50!black,-latex] (vert020.south)--(vert010.north);
        
\draw[style=semithick, dotted,green!50!black,-latex] (vert021.south)--(vert011.north);
        
 \draw[style=semithick, dotted,green!50!black,-latex] (vert120.south)--(vert110.north);
        
\draw[style=semithick, dotted,green!50!black,-latex] (vert121.south)--(vert111.north);
        
\draw[style=semithick, dotted,green!50!black,-latex] (vert220.south)--(vert210.north);
        
\draw[style=semithick, dotted,green!50!black,-latex] (vert221.south)--(vert211.north);
        
\draw[style=semithick, blue,-latex] (vert111.west)--(vert011.east);
        
\draw[style=semithick,blue,-latex] (vert121.west)--(vert021.east);
        
\draw[style=semithick,blue,-latex] (vert211.west)--(vert111.east);
        
\draw[style=semithick, blue,-latex] (vert221.west)--(vert121.east);

\draw[style=semithick,       dashed,red,latex-] (vert001.north east) -- (vert000.south west);
        
\draw[style=semithick, dashed,red,latex-] (vert011.north east) -- (vert010.south west);
        
\draw[style=semithick, dashed,red,latex-] (vert021.north east) -- (vert020.south west);
        
\draw[style=semithick, dashed,red,latex-] (vert101.north east) -- (vert100.south west);
        
\draw[style=semithick, dashed,red,latex-] (vert111.north east) -- (vert110.south west);
        
\draw[style=semithick,dashed,red,latex-] (vert121.north east) -- (vert120.south west);
        
\draw[style=semithick, dashed,red,latex-] (vert201.north east) -- (vert200.south west);
        
\draw[style=semithick, dashed,red,latex-] (vert211.north east) -- (vert210.south west);
        
\draw[style=semithick, dashed,red,latex-] (vert221.north east) -- (vert220.south west);

\node at (3,1) {$\Omega_{3,(2,1,2)}$};

\end{scope}

\end{tikzpicture}
\]

\noindent
The $1$-skeleton of $\Omega_{2,(5,3)}$ is $(E_{2,(5,3)},c)$, where $(E_{2,(5,3)},c)$ is defined in Definition~\ref{dfn:Ekm}.
\end{enumerate}
\end{examples}

%
%

\begin{dfns}
The $k$-graph  $\Lambda$ is \emph{connected} if the equivalence relation $\sim$ on $\Lambda^0$ generated by $\{(u,v)\mid u\Lambda v\ne\emptyset\}$ is $\Lambda^0\times\Lambda^0$.
A $k$-graph $\Lambda$ is \emph{singly connected} if there is at most one path between any two  vertices; that is, for all $u,v \in \Lambda^0$ we have $| u {\Lambda}v | \le 1$.
\end{dfns}
\begin{minipage}[l]{0.8\textwidth}
\begin{dfn}[{Locally convex, \cite[Definition 3.9]{RSY1}}]
A rank-$k$ graph $\Lambda$ is \textit{locally convex} if for $e \in \Lambda^{\varepsilon_i}$
and $f \in \Lambda^{\varepsilon_j}$, $1 \le i \neq j \le k$ such that $r(e) = r(f)$, there are $e' \in s ( f ) \Lambda^{\varepsilon_i}$ and $f' \in s(e)\Lambda^{\varepsilon_j}$ such that $ef'=fe'$.
\end{dfn}
\end{minipage}%
\begin{minipage}[r]{0.2\textwidth}
\[
\begin{tikzpicture}[scale=0.6]

 \node[inner sep=2.8pt, circle] (27) at (0,0) {$\bullet$};
\node[inner sep=2.8pt, circle] (28) at (0,2) {$\bullet$};	
\node[inner sep=2.8pt, circle] (29) at (2,0) {$\bullet$};
\node[inner sep=2.8pt, circle] (30) at (2,2) {$\bullet$};	
\path[->, red, dashed, >=latex,thick] (28) edge [left] node {$f$} (27);
\path[->, red, dashed, >=latex,thick] (30) edge [right] node {$f'$} (29);

\path[->, blue, >=latex,thick] (29) edge [below] node {$e$} (27);
\path[->, blue, >=latex,thick] (30) edge [above] node {$e'$} (28);

\end{tikzpicture}
\]
\end{minipage}

\medskip
\noindent
The locally convex condition ensures that the range of the degree map $d : \Lambda \to \NN^k$ is a convex set of $\NN^k$. \par
\medskip
\textbf{Standing assumption:} We shall assume that all $k$-graphs in this manuscript are locally convex. \par
\medskip
\noindent
For a locally convex $k$-graph $\Lambda$ and 
$m \in \mathbb{N}^k$, we write 
\[
\Lambda^{\leq m} := \{ \lambda \in \Lambda : d(\lambda) \leq m 
\text{ and } s ( \lambda ) \Lambda^{\varepsilon_i} = \emptyset 
\text{ whenever } d  (\lambda ) + \varepsilon_i \leq m \} .
\]

\noindent
Note that $\Lambda^{\le 0} = \Lambda^0$.

\begin{dfn}\label{def:essential}
Let  $\Lambda$ be a rank-$k$ graph, $G$ a countable group, and $c : \Lambda \to G$ a functor (1-cocycle), then $c$ is \emph{essential} if the restriction of $c$ to $u{\Lambda}v$ is injective for all $u, v \in \Lambda^0$.
\end{dfn}

\begin{example}
It is straightforward to check that the rank-$k$ graph $\Omega_{k,m}$ described Example~\ref{ex:omega} is connected and singly connected. The degree functor is essential.
\end{example}

\noindent
Every rank-$k$ graph $\Lambda$ has a fundamental groupoid, defined as follows (see \cite[Section~19.1]{Schubert} or \cite[Section~3]{PQR1}). The following definitions come from \cite{bkps}:

\begin{dfn}\label{dfn:fundgpd}
Let $\Lambda$ be a rank-$k$ graph. There exists a groupoid $\Pi(\Lambda)$ and a functor $i : \Lambda \to \Pi(\Lambda)$ such that $i(\Lambda^0)=\Pi(\Lambda)^0$, with the following universal property: 

\begin{minipage}[l]{0.7\textwidth}
For every functor $F$ from $\Lambda$ into a groupoid $\Gg$, there exists a unique groupoid homomorphism $\tilde{F} : \Pi(\Lambda) \to \Gg$ such that $\tilde{F} \circ i = F$. The pair $(\Pi(\Lambda), i)$ is unique up to canonical isomorphism, so we refer to any such groupoid $\Pi(\Lambda)$ as the \emph{fundamental groupoid} of $\Lambda$.
\end{minipage}%
\begin{minipage}{0.3\textwidth}
\[
\begin{tikzpicture}

\node[inner sep=0pt,circle] at (2,2) {$\Pi ( \Lambda )$};
\node at (0,2) {$\Lambda$};
\node at (2,0) {${\mathcal G}$};
\node at (1,2.2) {$i$};
\node at (0.8,0.8) {$F$};
\node at (2.25,1) {$\tilde{F}$};

\draw[->] (0.3,2) -- (1.5,2);
\draw[->] (2,1.8) -- (2,0.2);
\draw[->] (0.2,1.8) -- (1.8,0.2);
\end{tikzpicture}
\]
\end{minipage}
\end{dfn}

\noindent
The assignment $\Lambda \mapsto \Pi (\Lambda)$ is a functor from $k$-graphs to groupoids. Note that $\Pi(\Lambda)$ is denoted $\Gg (\Lambda)$ in \cite{PQR1}. Each $k$-graph also has a fundamental group; the standard definition, for connected rank-$k$ graphs, is as any one of the isotropy groups of its fundamental groupoid, as follows.

\begin{dfn}
Let $\Lambda$ be a rank-$k$ graph, then the \emph{fundamental group} $\pi_1(\Lambda, v)$ of $\Lambda$ at $v \in\Lambda^0$ is the isotropy group $\pi_1 (\Lambda, v) := v \Pi ( \Lambda ) v$ of $\Pi(\Lambda)$ at $v$.
\end{dfn}
\begin{rmk} \label{rmk:opposite0}
Let $(\Lambda,d)$ be a rank-$k$ graph, then the opposite category $\Lambda^{\mathrm{op}}$ with the degree functor $d_{\mathrm{op}} ( \lambda^{\mathrm{op}} ) = d(\lambda)$ is also a rank-$k$ graph. It is straightforward to see that $\Pi ( \Lambda ) = \Pi ( \Lambda^{\mathrm{op}} )$ and hence $\pi_1 (\Lambda, v) = \pi_1 ( \Lambda^{\mathrm{op}},v)$. The $1$-skeleton of $\Lambda^{\mathrm{op}}$ is $\Sk_\Lambda^{\mathrm{op}}$, that is the directed graph $( \Lambda^0 , \Lambda^{\varepsilon_1} \sqcup \Lambda^{\varepsilon_2} , s , r)$.
\end{rmk}

\subsection{Gluing higher-rank graphs} \label{sec:glue-coloured}

\noindent
There are results about gluing higher-rank graphs in   written in categorical terminology in \cite[\S 2]{KPSW}. Since we are using coloured graphs with bi-coloured commuting squares to model higher-rank graphs exclusively here, we pause briefly here to give these results from a coloured graph perspective. 

Let $(E,c)$ be a $k$-coloured graph with squares $\Cc$.  
Suppose that $\sim_1$ is an equivalence relation on $E^0 \sqcup E^1$
with the following property: 
\begin{center}
For $e,e' \in E^1$ we have if $e \sim_1 e'$ then $c(e) = c(e')$, $s(e) \sim_1 s(e')$, $r(e) \sim_1 r(e')$. 
\end{center}

\begin{rmk} \label{rmk:simdef}
We may expand the equivalence relation $\sim_1$ on $E^0 \sqcup E^1$, to a relation $\sim$ on all of $E^*$ as follows:  First we define a relation $\sim_2$ on $\sqcup_{i=0}^2 E^i$ as follows: 
\begin{equation} \label{eq:nested1}
x \sim_2  y \text{ if and only if } \begin{cases} x=ef, y=e'f' & e \sim_1 e', f \sim_1 f', \\
x \sim_1 y & \text{ otherwise }. \\
\end{cases}
\end{equation}

\noindent
Note that $c(e)=c(e')$ and $c(f)=c(f')$ is automatic by definition of $\sim_1$. Then for $n \ge 3$ we define a relation $\sim_n$ on $\sqcup_{i=0}^n E^i$ as follows: $x \sim_n  y$ if and only if either $x=\alpha ,y=\alpha'$ with 
\begin{equation} \label{eq:nested2}
\alpha_1\cdots \alpha_{n-1} \sim_{n-1} \alpha_1 \cdots \alpha_{n-1}',  \text{ and } \alpha_2 \cdots \alpha_n \sim_{n-1} \alpha_2' \cdots   \alpha_n', \text { or } \alpha \sim_{n-1} \alpha'  .
\end{equation}

\noindent
This leads us to a relation $\sim$ on $E^*$, whose properties we summarise below.
\end{rmk}

\begin{lem}[Quotient graph] \label{lem:quotient k-colouredgraph}
Let $(E,c)$ be a $k$-coloured graph with commuting squares $\Cc$.  
Suppose that $\sim_1$ is an equivalence relation on $E^0 \sqcup E^1$
with the following properties: For $e,e' \in E^1$ if $e \sim_1 e'$ then $c(e) = c(e')$, $s(e) \sim_1 s(e')$, $r(e) \sim_1 r(e')$. Then the relation $\sim$ on $E^*$ described in Remark~\ref{rmk:simdef} above is an equivalence relation.

Set $E/{\Tiny\sim}  = (E^0/{\Tiny\sim},E^1/{\Tiny\sim},r,s)$, $c([e])=c(e)$ where $r([e])=[r(e)]$, $s([e])=[s(e)]$. Then $( E/{\Tiny\sim} , c )$ is a $k$-coloured graph with squares $\Cc / {\Tiny\sim}= \{ [e][f] {\Tiny\sim} [g][h] : ef \sim_\Cc gh \}$.
\end{lem}

\begin{proof}
That $\sim$ is an equivalence relation on $E^*$ follows by showing that each $\sim_n$, $n \ge2$ is an equivalence relation on $\sqcup_{i=0}^n E^i$. The rest is a straightforward calculation.

It follows that if $x \sim_1 y$ for $x,y \in E^0 \sqcup E^1$ then either $x,y \in E^0$ or $x,y \in E^1$ as the domain of $c$ is $E^1$, and $r(v)=s(v)=v$ if $v \in E^0$. Let $[e]_1$ denote the equivalence class containing $e \in E^1$ and suppose that $e' \in [e]_1$. Then $c(e) = c(e')$, $s(e) \sim_1 s(e')$, $r(e) \sim_1 r(e')$. Hence, the formulas
\[
c_1([e]_1) := c(e), \quad r([e]_1) = [r(e)]_1,\text{ and } s([e]_1) = [s(e)]_1
\]

\noindent
are well defined and so the quotient coloured graph $E/{\Tiny\sim_1}  = (E^0/{\Tiny\sim_1},E^1/{\Tiny\sim_1},r,s)$, $c_1([e]_1)=c(e)$ makes sense. One checks
\[
(E/{\Tiny\sim} ,c) = (E^0/{\Tiny\sim},E^1/{\Tiny\sim},r,s) ~=~ (E/{\Tiny\sim_1},c_1)   = (E^0/{\Tiny\sim_1},E^1/{\Tiny\sim_1},r,s)
\]

\noindent
where $c([e])=c(e)=c([e]_1)$ as the equivalence relations $\sim_n$, $n \ge 2$ are nested (cf. \eqref{eq:nested1}, \eqref{eq:nested2}). 

Suppose $ef \sim_\Cc gh$ in $(E,c)$, then we claim that $[e][f] \sim_\Cc [g][h]$ in $(E/{\Tiny\sim},c)$. By \eqref{eq:nested1} we have $[ef]_2 = [e]_1 [f]_1$, and similarly $[fg]_2 = [f]_1 [g]_1$. Then $[e]_1 [f]_1 = [ef]_2 = [gh]_2 = [g]_1 [h]_1$ and so $[e]_1 [f]_1 \sim_\Cc [g]_1 [h]_1$, and the result follows by the nesting property of $\sim$.
\end{proof}

\begin{nota}
We shall refer to $(E/{\Tiny\sim} ,c)$ as the \emph{quotient $k$-coloured graph} of generated by the equivalence relation $\sim_1$ with quotient squares $\Cc / 
{\Tiny\sim}$.
\end{nota}

\begin{dfn}[Hereditary and co-hereditary subgraphs]
Let $E=(E^0,E^1,r,s)$ be a directed graph with $E^0,E^1$ finite. A locally convex subgraph $F=(F^0,F^1,r,s)$ is \emph{hereditary} if every finite path which enters $F$ cannot leave $F$, that is $r^{-1} ( F^1) \subset F^0$. A locally convex subgraph $F=(F^0,F^1,r,s)$ is \emph{co-hereditary} if every finite path which leaves $F$ cannot return to $F$, that is $r^{-1} ( E^1 \backslash F^1) \subset E^0 \backslash F^0$.
\end{dfn}

\begin{example}
Consider the graph $q_{al}$ in \eqref{eq:gluex} below. The locally convex subgraph $F^0 = \{ L(a) , r(a) \}$, $F^1 = \{ ( L(a),a) \}$, with inherited range and source maps is both hereditary and co-hereditary.

Now consider the graph $q_{al} \# q_{ar}$ in \eqref{eq:gluex} below. The locally convex subgraph $F$ with $F^0 = \{ R'(a), r(a), L(a),a \}$, $F^1 = \{ (r(a),R'(a)), (L(a),a), (s(a),L(a)) , (a,R'(a)) \}$, with inherited range and source maps is hereditary as it takes in the sink $a$ of $q_{al} \# q_{ar}$.
\end{example}

\noindent
The following is an analogue of \cite[Example 2.2]{KPSW}.

\begin{thm}[Gluing hereditary and co-hereditary $k$-coloured graphs] \label{thm:glue}
Let $(E_1,c_1)$, $(E_2,c_2)$ be two $k$-coloured graphs with squares $\Cc_1$, $\Cc_2$ respectively. Suppose $F_1$ (resp.\ $F_2$) are  locally convex subgraphs of $E_1$ (resp.\ $E_2$) with squares $\Cc^F_1$ (resp.\ $\Cc^F_2$). Let $\phi : F_1 \to F_2$ be a $k$-coloured graph isomorphism which sends squares in $(F_1,c_1)$ to squares in $(F_2,c_2)$. Let $G = ( E_1 \sqcup E_2 , c_1\sqcup c_2)$ be $k$-coloured graph which consists of the disjoint union of $E_1 , E_2$ induced colouring function and squares $\Cc_1 \sqcup \Cc_2$. Define a  relation $\sim_1$ on $G^0 , G^1$ by 
\begin{enumerate}
\item $v \sim_1 w$ if and only if there are $v' \in F_1^0$, $w' \in F_{2}^0$ such that $\phi (v') = v$ and  $\phi ( w' ) = w \in F_2^0$.
\item $e \sim_1 f$ if and only if $c(e)=c(f)$, and there are $e' \in F_2^1$, $f' \in F_{2}^1$ such that $\phi (e) = e'$ and $\phi ( f) =f'$.
\end{enumerate}
Then we have the following
\begin{enumerate}[label={(\roman*)}]
\item The relation $\sim_1$ is an equivalence relation and gives rise to an equivalence relation $\sim$, as defined in  Lemma~\ref{lem:quotient k-colouredgraph}, is an equivalence relation on $G^*$ and $( G/{\Tiny\sim} , c_1 \sqcup c_2 )$ is a $k$-coloured graph with squares $(\Cc_1 \sqcup \Cc_2) / {\Tiny\sim}= \{ [e][f] {\Tiny\sim} [g][h] : ef \sim_{\Cc_1} gh \text{ or } ef \sim_{\Cc_2} gh \}$.

 \item Suppose that $F_1, F_2$ are both hereditary (resp.\ co-hereditary) subgraphs of $E_1$ (resp.\ $E_2$) and that $\Cc_1 , \Cc_1 |_{F_1}$ (resp.\ $\Cc_2 , \Cc_2 |_{F_2}$) are complete sets of squares in $(E_1,c_1)$ (resp.\ $(E_2,c_2)$). Then the collection of squares $(\Cc_1 \sqcup \Cc_2) / {\Tiny\sim}= \{ [e][f] {\Tiny\sim} [g][h] : ef \sim_{\Cc_1} gh \text{ or } ef \sim_{\Cc_2} gh \}$ in $( G/{\Tiny\sim} , c_1 \sqcup c_2 )$ is complete.
\end{enumerate}
\end{thm}

\begin{proof}
For \emph{(i)} we begin by observing that $\sim_1$ is an equivalence relation satisfying the hypotheses of Lemma~\ref{lem:quotient k-colouredgraph} as $\phi$ is a graph morphism and so  $s(e) \sim_1 s(e')$, $r(e) \sim_1 r(e')$. Hence \emph{(i)} holds by applying this result.

For \emph{(ii)} shall concentrate on the arguments for the hereditary case, leaving the co-hereditary case out as it will be similar. Consider $gh \in G_2^2$. Without loss of generality suppose $h \in E_2^1$. Now we must break into two cases: First $h \in F_2^1$, second $h \not \in F_2^1$ (recognising that one of these cases may be vacuous). 

Firstly, as $F_2$ is hereditary $g \in F_2^1$, and then $gh \in F_2^2$.
Since $\Cc_2 |_{F_2}$ is complete there is $e'f' \in F_2^2$ such that $e'f' \sim_{\Cc_2} gh$. Since $\phi$ is bijective and square preserving, there is $ef \in F_1^2$ such that $\phi (ef)=e'f'$. By definition of $\sim$ we have $[e'][f'] \sim [g][h]$ where $[e'][f'] = [e][f] \sim_{\Cc_2} [g][h]$. Secondly, we have $gh \in E_2^2$.
Since $\Cc_2$ is complete there is $ef \in E_2^2$ such that $ef \sim_{\Cc_2} gh$. By definition of $\sim$ we have $[e][f] \sim [g][h]$ where $[e][f] \sim_{\Cc_2} [g][h]$.
\end{proof}

\begin{example}
In section~\ref{sec:two} we shall meet the $2$-coloured graphs $(q'_{al},c_1)$ and $(q_{ar},c_2)$ each with one commuting square. We wish to glue them along the hereditary subgraphs $F',F$ consisting of the edges $(L(a'),a')$, $(L(a),a)$ respectively and the accompanying range and source vertices, to form the $2$-coloured graph $(q'_{al} \# q_{ar} , c_1 \sqcup c_2 )$ as shown below. Indeed, this graph may then be then be glued along the edges $(r(a'),R'(a'))$, $(r(a'),L(a'))$, $(s(a),L(a))$ and $(s(a),R(a))$ to a related ``mirror" graph, to form the graph found in \eqref{eq:C2uncoloured}.
\begin{equation} \label{eq:gluex}
\begin{array}{c}
\begin{tikzpicture}[->=stealth,xscale=0.68,yscale=0.78]

\begin{scope}[scale=0.9,xshift=-10.6cm]

\node[inner sep=1pt,circle] (p0) at (-3.2,0) {\Tiny $r(a')$};
\node[inner sep=1pt,circle] (f0) at (-0.2,0) {\Tiny $L(a')$};
\node[inner sep=1pt,circle] (lf1) at (-6.2,0) {\Tiny $R(a')$};

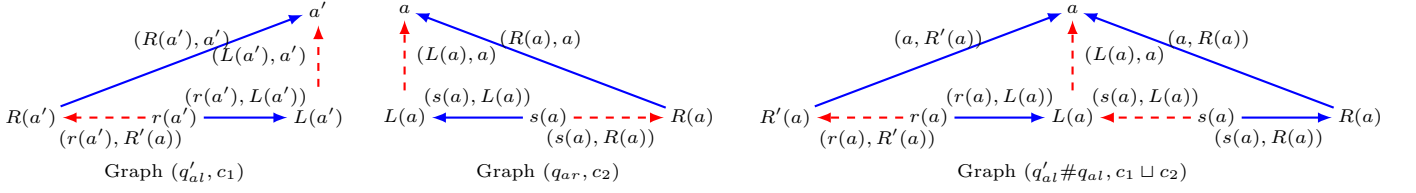
\node[inner sep=1pt,circle] (a0) at (-0.2,2) {\Tiny $a'$};
\draw[thick,-latex,blue] (p0) -- (f0) node[pos=0.5,above] {\Tiny\color{black}$(r(a'),L(a'))$};
\draw[thick,-latex,red,dashed] (p0) -- (lf1) node[pos=0.3,below] {\Tiny\color{black}$(r(a'),R'(a))$};

\draw[thick,-latex,blue] (lf1) -- (a0) node[pos=0.5,above] {\Tiny\color{black}$(R(a'),a')$};
\draw[thick,-latex,red,dashed] (f0) -- (a0) node[pos=0.5,left] {\Tiny\color{black}$(L(a'),a')$};

\node at (-3,-1) {\Tiny Graph $(q'_{al},c_1)$};
\end{scope}

\begin{scope}[scale=0.88,xshift=-9.48cm]

\node[inner sep=1pt,circle] (p1) at (3,0) {\Tiny $s(a)$};
\node[inner sep=1pt,circle] (f0) at (0,0) {\Tiny $L(a)$};
\node[inner sep=1pt,circle] (rf1) at (6,0) {\Tiny $R(a)$};

\node[inner sep=1pt,circle] (a0) at (0,2) {\Tiny $a$};
\draw[thick,-latex,blue] (p1) -- (f0) node[pos=0.5,above] {\Tiny\color{black}$(s(a),L(a))$};
\draw[thick,-latex,red,dashed] (p1) -- (rf1) node[pos=0.3,below] {\Tiny\color{black}$(s(a),R(a))$};
\draw[thick,-latex,red,dashed] (f0) -- (a0) node[pos=0.5,right] {\Tiny\color{black}$(L(a),a)$};
\draw[thick,-latex,blue] (rf1) -- (a0) node[pos=0.5,above] {\Tiny\color{black}$(R(a),a)$};

\node at (3,-1) {\Tiny Graph $(q_{ar},c_2)$};
\end{scope}

\begin{scope}[xshift=4cm,scale=0.88]

\node[inner sep=1pt,circle] (p0) at (-3,0) {\Tiny $r(a)$};
\node[inner sep=1pt,circle] (p1) at (3,0) {\Tiny $s(a)$};
\node[inner sep=1pt,circle] (f0) at (0,0) {\Tiny $L(a)$};
\node[inner sep=1pt,circle] (lf1) at (-6,0) {\Tiny $R'(a)$};
\node[inner sep=1pt,circle] (rf1) at (6,0) {\Tiny $R(a)$};
\node[inner sep=1pt,circle] (a0) at (0,2) {\Tiny $a$};

\draw[thick,-latex,blue] (p0) -- (f0) node[pos=0.5,above] {\Tiny\color{black}$(r(a),L(a))$};
\draw[thick,-latex,red,dashed] (p0) -- (lf1) node[pos=0.3,below] {\Tiny\color{black}$(r(a),R'(a))$};
\draw[thick,-latex,red,dashed] (p1) -- (f0) node[pos=0.5,above] {\Tiny\color{black}$(s(a),L(a))$};
\draw[thick,-latex,blue] (p1) -- (rf1) node[pos=0.3,below] {\Tiny\color{black}$(s(a),R(a))$};
\draw[thick,-latex,blue] (lf1) -- (a0) node[pos=0.5,above] {\Tiny\color{black}$(a,R'(a))$};
\draw[thick,-latex,dashed,red] (f0) -- (a0) node[pos=0.5,right] {\Tiny\color{black}$(L(a),a)$};
\draw[thick,-latex,blue] (rf1) -- (a0) node[pos=0.5,above] {\Tiny\color{black}$(a,R(a))$};

\node at (0,-1) {\Tiny Graph $(q'_{al} \# q_{al},c_1 \sqcup c_2)$};

\end{scope}

\end{tikzpicture}
\end{array}
\end{equation}
\end{example}

\noindent
Conversely we may ``cut-off" hereditary and cohereditary locally convex subgraphs from a $k$-coloured graph. 

\begin{rmk}
Suppose that  $(F,c)$ is a hereditary subgraph of $(E,c)$ and $(fe,e'f')$ is a commuting square in $E$, then if $s(e) \in F^0$ then $e,f,e'f \in F^1$ similarly if $(F,c)$ is a cohereditary subgraph of $(E,c)$ and $s(e) \not\in F^0$ then $e,f,e',f' \in E^1 \backslash F^1$.
\end{rmk}

\begin{thm}
Let $(E,c)$ be a $k$-coloured graph with complete set of squares $\Cc$ and $(F,c)$ be a hereditary (resp. cohereditary) subgraph. Let $(E\backslash F ,c) := ( ( E^0 \backslash F^0 , E^1 \backslash F^1 ,r,s) , c)$,
\[
 \Cc |_{E\backslash F} = \{ (fe,e'f') \in \Cc : s(e) \notin F^0 \}  , \ \Cc |_{F}
= \{ (fe,e'f') \in \Cc : s(e) \in F^0 \}
\]

\noindent
Then $(F, c) , \Cc |_{F}$ is a $k$-coloured graph with a complete set of squares (resp. $(E \backslash F, c), \Cc |_{E \backslash F}$ is a $k$-coloured graph with a complete set of squares).
\end{thm}

\begin{proof}
In the case $(F, c)$, with commuting squares $\Cc |_{F}$,  the commuting squares are exactly the ones with vertices within $F$ and since $\Cc$ is complete for $E^0$ it follows that $ \Cc |_{F}$ is complete for $F^0$. A similar argument applies for $(E \backslash F, c)$, with commuting squares $\Cc |_{E \backslash F}$.
\end{proof}

\section{Quadrangle clubs built from polyhedral graphs} \label{sec:two}

\subsection{Convex polyhedra and polyhedral graphs} 

Roughly speaking, a \emph{convex polyhedron} and a \emph{polyhedral graph} are one and the same thing, but in a different place. The first of which is an undirected graph embedded in the sphere and the second is a planar graph derived  from the former by stereographic projection (clearly this process is reversible). If the graph embedded in the sphere is a tiling of the sphere, then it is called a \emph{spherical polyhedron}.

In the diagram below we have the convex polyhedron with points $v_1,v_2$ on the equator together with arcs $a_0$ round the left equator, $a_2$ round the right equator and $a_1$ round the south pole. This is a spherical polyhedron and maps via stereographic projection to the polyhedral graph, $C_2$, shown on the right.
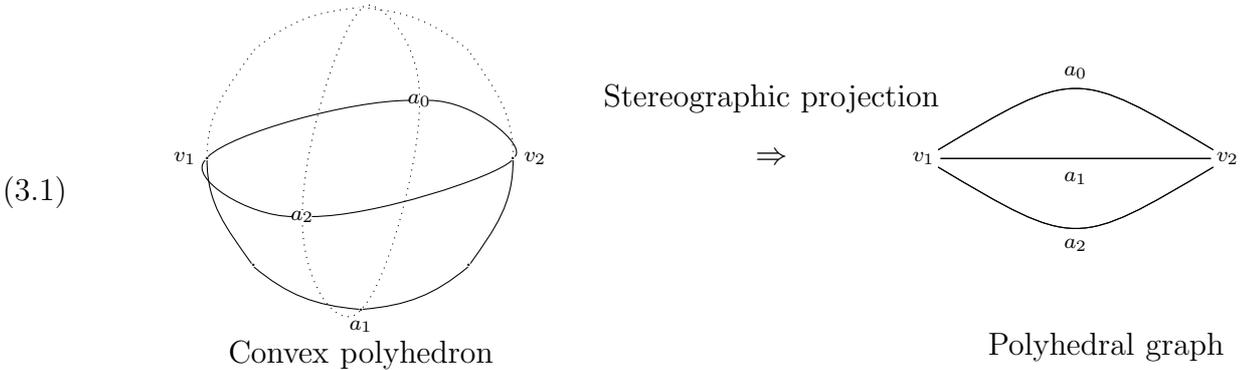
\begin{equation} \label{eq:conv-poly}
\begin{array}{l}
\begin{tikzpicture}[scale=2,->=stealth]

\node[inner sep= -1pt] (w1) at (xyz spherical cs:longitude=0,latitude=0,radius=1) {\tiny $ $};
    \node[inner sep= -1pt] (y0) at (xyz spherical cs:longitude=45,latitude=0,radius=1) {\small$  $};
    \node[inner sep= -1pt] (v1) at (xyz spherical cs:longitude=90,latitude=0,radius=1) {\tiny $.$};
    \node[inner sep= -1pt] (x0) at (xyz spherical cs:longitude=135,latitude=0,radius=1) {\tiny $.$};
    \node[inner sep= -1pt] (w0) at (xyz spherical cs:longitude=180,latitude=0,radius=1) {\tiny $.$};
    \node[inner sep= -1pt] (y1) at (xyz spherical cs:longitude=225,latitude=0,radius=1) {\tiny $.$};
    \node[inner sep= -1pt,left] (v0) at (xyz spherical cs:longitude=270,latitude=0,radius=1) {\tiny $.$};
    \node[inner sep= -1pt] (x1) at (xyz spherical cs:longitude=315,latitude=0,radius=1) {\tiny $  $};
    \node[inner sep= -1pt] (u1) at (xyz spherical cs:longitude=0,latitude=90,radius=1) {\tiny $a_2$};
    \node[inner sep= -1pt] (u0) at (xyz spherical cs:longitude=0,latitude=-90,radius=1) {\tiny $a_0 $};
    \draw[-,dotted, out=40, in=185] (x1.75) to (w1.west);
    \draw[- ,dotted,out=235, in=90] (x1.220) to (v0.north);
    \draw[-, out=125, in=270] (y1.140) to (v0.south);
    \draw[-, out=320, in=175] (y1.285) to (w0.west);
    \draw[-, out=220, in=5] (x0.255) to (w0.east);
    \draw[-, out=55, in=270] (x0.40) to (v1.south);
    \draw[-,dotted, out=140, in=355] (y0.105) to (w1.east);
    \draw[white, line width=2pt] (w1.255) .. controls +(0,0,0.4) and +(0,0.4,0) .. (u1.north);
    \draw[-,dotted] (w1.255) .. controls +(0,0,0.4) and +(0,0.4,0) .. (u1.north);
    \draw[-,dotted] (w0.240) .. controls +(0,0,0.4) and +(0,-0.4,0) .. (u1.south);
    \draw[-,dotted] (w0.45) .. controls +(0,0,-0.4) and +(0,-0.4,0) .. (u0.south);
    \draw[-,dotted] (w1.20) .. controls +(0,0,-0.2) and +(0,0.4,0) .. (u0.north);
    \draw[-,dotted, out=305, in=90] (y0.320) to (v1.north);

    \draw[-] (v0.60) .. controls +(0,0,-0.4) and +(-0.4,0,0) .. (u0.west);
    \draw[-] (v0.225) .. controls +(0,0,0.3) and +(-0.3,0,0) .. (u1.west);
    \draw[-] (v1.240) .. controls +(0,0,0.3) and +(0.4,0,0) .. (u1.east);
    \draw[-] (v1.45) .. controls +(0,0,-0.2) and +(0.3,0,0) .. (u0.east);
    \node[left] at (v0) {\tiny $v_1$};
    \node[below] at (w0) {\tiny $a_1$};
    \node[right] at (v1) {\tiny $v_2$};
    
\node at (0,-1.3) {Spherical polyhedron};

\begin{scope}[xshift=3.7cm,yshift=0cm]

\node at (-1,0.4) {Stereographic projection};
\node at (-1,0) {$\Rightarrow$};

\node[inner sep=1pt, circle] (00) at (0,0) {\tiny $v_1$};
\node[inner sep=1pt, circle] (10) at (2,0) {\tiny $v_2$};

\draw[-] (00)--(10) ;
\draw[-] (00)--(10);
\draw[-,dotted] (1,0.35)--(2.5,0);
\draw[-,dotted] (1,-0.35)--(2.5,0);
\draw[-,dotted] (1,-0.35)--(1,0.35);
     
\draw[-] (10) .. controls +(-1,+0.6) .. (00);
\draw[-] (10) .. controls +(-1,+0.6) .. (00);

\node[below] at (1,-0.46) {\tiny\color{black} $a_2$};
\node[above] at (1,0.46) {\tiny\color{black} $a_0$};
\node[below] at (0.7,-0.02) {\tiny\color{black} $a_1$};
     
\draw[-] (10) .. controls +(-1,-0.6) .. (00);
     
\draw[-] (10) .. controls +(-1,-0.6) .. (00);
     
\node at (1.2,-1.25) {Polyhedral graph};
\node[inner sep=3pt] at (1,-0.35) {\Tiny $r_1$};
\node[inner sep=3pt] at (1,0.35) {\Tiny $r_2$};
\node[inner sep=3pt] at (2.5,0) {\Tiny $r_0$};
\end{scope}

\end{tikzpicture}
\end{array}
\end{equation}

\noindent
Our starting point is, as in \cite[\S 1]{Lambek}, a convex  polyhedron $\mathfrak{P}=(P,A)$ (called Eulerian in \cite{Lambek}) described by finite sets of points and arcs $(P, A)$, inscribed on the sphere $S^2$. Since the graph is embedded on a sphere there is no distinction between the bounded faces and the unbounded face of the graph. So the set $P$ is a finite set $p$ of \emph{points} on the sphere, and $A$ is a collection of mutually disjoint subsets $a$ of $S^2 \setminus P$, called \emph{arcs}.  Each arc is homeomorphic to the interval $(0,1)$ such that for each $a \in A$ there are distinct $r(a), s(a) \in P$ such that $\overline{a} = a \cup \{r(a), s(a)\}$ (that is $\mathfrak{P}$ has no loops/cycles), and no other arc $b \neq a$ satisfies $\overline{b}=b \cup \{r(a), s(a)\}$ (that is $\mathfrak{P}$ has no parallel edges). Hence we may identify $A$ with $\{ (v,w) \in P^2 : r(a) =v \neq  w=s(a) \text{ some } a \in A \}$. For each $a \in A$ the points $r(a), s(a) \in P$ are said to be \emph{neighbours}. The interior of each arc contains no vertex and no point of an other arc (that is $\mathfrak{P}$ is planar). The \emph{valency} (or degree) $d_{\mathfrak{P}} (v)$ of a point $v \in P$ in the graph $\mathfrak{P}=(P,A)$ is the number of neighbours of $v$.

The complement of $\cup_{p\in P} p \cup \bigcup_{a \in A} a$ on the sphere then consists of $|A| - |P| + 2$ distinct open sets which we call the \emph{faces}. We write $F$ for the set of faces and typically denote its elements by $f$. By hypothesis we have the following three adjacencies: \emph{an arc $a$ is adjacent to a face} $f$ if $a \subseteq \overline{f}$, \emph{an arc $a$ is adjacent to a point} $p$ if $p \in \overline{a}$, and \emph{a point $p$ is adjacent to a face} $f$ if $p \in \overline{f}$. 

\textbf{Standing assumption:} We assume that $\mathfrak{P}$ is connected, and that every point is adjacent to two arcs (that is $\mathfrak{P}$ has no sinks/sources).  For more information on planar graphs see \cite[\S4]{D}.

\noindent
We shall need to use some of the methods of Johnstone \cite{Johnstone} found in his paper on embeddings of categories in a groupoid. Since we will not require the full generality of Johnstone's set up, we describe a specific arrangement  which he calls a quadrangle club, which arises from a Lambek diagram as in \cite{Lambek}. The specific quadrangle clubs we consider are the ones described in \cite[\S 6]{Johnstone}.

\subsection{Quadrangle clubs and half-arc Lambek conditions} \label{sec:Plabel}

Johnstone's method for constructing a quadrangle club involves placing a vertex in the middle of each arc to produce two \emph{half-arcs}: For each arc $a$ in the convex polyhedron, this produces two half arcs called $(r(a),a)$ and $(a,s(a))$ by dividing the arc $a$ in two, with ``midpoint" labelled $a$. Each half-arc maintains the face-adjacency ($L=$ Left,$R=$ right) of $a$, that is $L((r(a),a) = L(a)= L((a,s(a))$ and $R((r(a),a) = R(a) = R((a,s(a))$(see \eqref{eq:half-arcs}).

\begin{dfn}[Lambek quadrangle club]
Let $\mathfrak{P}=(P, A)$ a polyhedral graph described as above, define 
\begin{equation} \label{eq:lqc}
E^0 = P \sqcup A \sqcup F \ \text{ and } \
E^1 = \{(a, f) \in A \times F : a \in \overline{f}\} \cup \{(f, p) \in F \times P : p \in \overline{f}\}.
\end{equation}

\noindent
The range and source maps are defined by $r(x,y) = x$ and $s(x,y) = y$. That is, there is an edge in $E^1$ from each point in $P$ to each face to which it is adjacent, and there is an edge in $E^1$ from each face in $F$ to each arc adjacent to it. This graph $E = E(\mathfrak{P}) = (E^0, E^1, r, s)$ is the \emph{Lambek quadrangle club} associated to $\mathfrak{P} = (P, A)$. (This determines a Lambek quadrangle club in the sense of \cite[Definition~1.2]{Johnstone}, if we colour elements of $P$ blue, elements of $F$ red and elements of $A$ green, and we colour elements of $(A \times F) \cap E^1$ purple, and elements of $(F \times P) \cap E^1$ yellow).
\end{dfn}

\noindent
Note that $E$ is a directed graph with $E^3 = \emptyset$, that is there are no paths of length three in $E$. Our standing assumption entails that $E$ has no cycles/loops, parallel edges or sources/sinks.

\begin{lem}
Let $\mathfrak{P}=(P,A)$ be a polyhedral graph and $E$ its associated Lambek quadrangle club, then $E$ is planar. 
\end{lem}

\begin{proof}
To confirm planarity we must consider the undirected version $\bar{E}$ of $E$, in which every edge $e \in E^1$ has an opposite $\bar{e}$ with the opposite  direction. Consideration of the range map on $E$ shows that $r : F \to A$ and $r : P \to F$. A moment's thought will show that, in the undirected case, all cycles in $\bar{E}$ have even length. Hence $\bar{E}$ cannot contain the Peterson graph $K_5$ as a subgraph. Next we claim that $\bar{E}$ is not bipartite. Consideration of the source map on $E$ shows that $s : A \to F$ and $s : F \to P$. Suppose, for contradiction, that $\bar{E}$ is bipartite, then $E^0 = U \sqcup V$. From the properties of the range and source maps we have  $F \subset U$ and $F \subset V$, a contradiction. Hence $\bar{E}$ is not bipartite and so does not contain the complete bipartite graph $K_{3,3}$. Therefore by Kuratowski's theorem 
\cite[Th\'{e}or\`{e}me B]{Kur} $\bar{E}$, and hence $E$ is planar.
\end{proof}

\noindent
Following \cite[Definition 1.1(3)]{Johnstone} we have:

\begin{dfn}[Quadrangle] \label{dfn:addquads}
A \emph{quadrangle} $q$ of a Lambek quadrangle club $E=(E^0, E^1, r, s)$ as above is a subgraph parameterised by a half-arc of $\mathfrak{P}=(P,A)$ with four vertices: one from $P$, two from $F$ and one from $A$ and four edges derived from the half arc.
\end{dfn}

\noindent 
Each arc $a \in A$ gives rise to two quadrangles $q_{ar}$ and $q_{al}$, from its two half-arcs as follows: 
\begin{itemize}
\item[$q_{ar}$:] Consists of the vertices $r(a), a , L(a), R(a)$, where $L(a)$ and $R(a)$ are the two faces to which the arc $a$ is adjacent. Here we let  $L(a), R(a)$ be determined by the orientation of the sphere. Another choice would give different directions. The four edges in the quadrangle connect these vertices as shown to the right below.  
\item[$q_{al}$:] Consists of the vertices $s(a), a, L(a), R(a)$. The four edges in the quadrangle connect these vertices as shown to the left below.
\end{itemize}
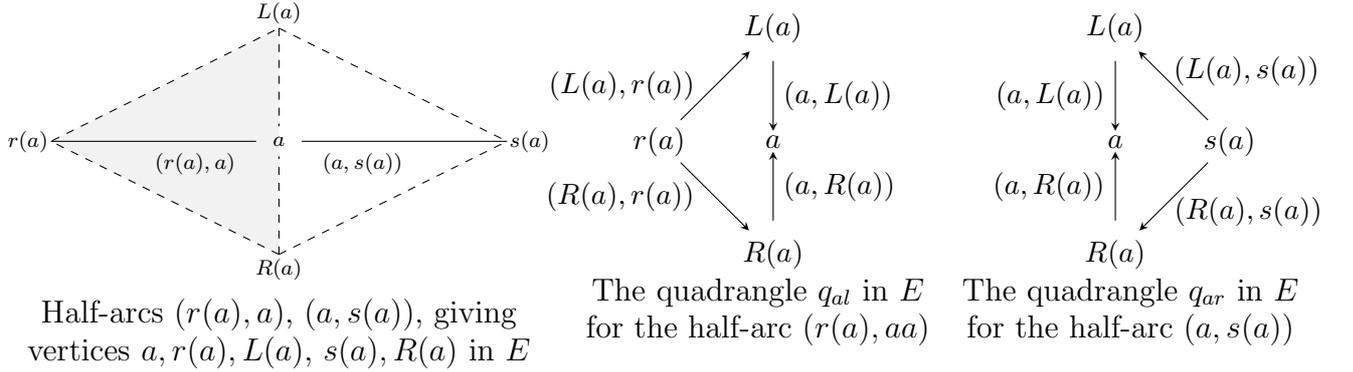
\begin{equation} \label{eq:half-arcs}
\begin{array}{ll}
&\begin{tikzpicture}

\filldraw[fill=gray!10,dashed] (0,-1.5)--(-3,0)--(0,1.5);

\node at (0,0) {\tiny $a$};
\node at (3.3,0) {\tiny $s(a)$};
\node at (-3.3,0) {\tiny $r(a)$};
\node at (0,1.7) {\tiny $L(a)$};
\node at (0,-1.7) {\tiny $R(a)$};
\node at (0,-2.3) {Half-arcs $(r(a),a)$, $(a,s(a))$, giving};
\node at (0,-2.8) {vertices $a,r(a),L(a)$, $s(a),R(a)$ in $E$};

\draw (-3,0) -- (-0.3,0);
\draw(0.3,0) -- (3,0);
\node at (-1.1,-0.3) {\tiny $(r(a),a)$};
\node at (1.1,-0.3) {\tiny $(a,s(a))$};

\draw[dashed] (0,1.5) -- (0,0.2);
\draw[dashed](0,-1.5) -- (0,-0.2);
\draw[dashed] (0,1.5) -- (3,0);
\draw[dashed] (0,-1.5) -- (3,0);

\begin{scope}[xshift=6.5cm]
\node at (-0.2,-2) {\text{The quadrangle $q_{a l}$ in $E$}};
\node at (-0.2,-2.5) {\text{for the half-arc $(r(a),a)$}};
\node at (0,-3) { };
\node[circle, inner sep=0pt] (t) at (0,1.5) {\small $L(a)$};
\node[circle, inner sep=0pt] (l) at (0,0) {\small $a$};
\node[circle, inner sep=0pt] (n) at (-1.5,0) {\small $r(a)$};
\node[circle, inner sep=0pt] (b) at (0,-1.5) {\small $R(a)$};
\draw[-stealth] (t)--(l) node[inner sep=0pt, pos=0.5, right] {\small $\ (a,L(a))  $};
\draw[-stealth] (n)--(t) node[inner sep=0pt, pos=0.5, left] {\small $(L(a),r(a)) \ \ $};
\draw[-stealth] (b)--(l) node[inner sep=0pt, pos=0.5, right] {\small $\ (a,R(a))  $};
\draw[-stealth] (n)--(b) node[inner sep=0pt, pos=0.5, left] {\small $(R(a),r(a)) \ \ $};

\end{scope}

\begin{scope}[xshift=11cm]

\node at (0.2,-2) {\text{The quadrangle $q_{a r}$ in $E$}};
\node at (0.2,-2.5) {\text{for the half-arc $(a,s(a))$}};
\node at (0,-3) { };
\node[circle, inner sep=0pt] (t) at (0,1.5) {\small $L(a)$};
\node[circle, inner sep=0pt] (l) at (0,0) {\small $a$};
\node[circle, inner sep=0pt] (r) at (1.5,0) {\small $s(a)$};
\node[circle, inner sep=0pt] (b) at (0,-1.5) {\small $R(a)$};
\draw[-stealth] (t)--(l) node[inner sep=0pt, pos=0.5, left] {\small $(a,L(a)) \ $};
\draw[-stealth] (r)--(t) node[inner sep=0pt, pos=0.5, anchor=south west] {\small $(L(a),s(a))$};
\draw[-stealth] (b)--(l) node[inner sep=0pt, pos=0.5, left] {\small $(a,R(a)) \ $};
\draw[-stealth] (r)--(b) node[inner sep=0pt, pos=0.5, anchor=north west] {\small $(R(a),s(a))$};

\end{scope}

\end{tikzpicture} \\
\end{array}
\end{equation}

\begin{dfn}[Half-arc Lambek conditions]
Let $\mathfrak{P}=(P,A)$ be a polyhedral graph and $E$ the associated Lambek quadrangle club. Fix $a \in A$. The \emph{half-arc Lambek conditions} for the half-arc pair $\{ (a,s(a)), (r(a),a) \}$ are equations $h_{al}, h_{ar}$ in the category $E^*$, formed from the quadrangles $q_{a l}, q_{a r}$ respectively as shown below
\begin{align} 
h_{a l} &:=  ~~ (a,R(a)) (R(a),r(a)) = (a,L(a)) (L(a),r(a)) ,  \text{ and } \label{eq:ha-left} \\
h_{a r} &:= ~~(a,R(a))  (R(a),s(a)) = (a,(L(a)) (L(a),s(a)) . \label{eq:ha-right}
\end{align}
\end{dfn}

\begin{rmks} \label{rmks:half-arc-count}
When the arcs in a polyhedral graph are presented as a enumerated list, say $A = \{ a_1 , \ldots , a_{|A|} \}$ then we use the subscript $i$ in place of $a_i$ in the name of the quadrangles $q_{il}, q_{ir}$ and half-arc Lambek conditions $h_{il}, h_{ir}$ associated to $a_i \in A$, $i=1, \ldots , |A|$.

There are $2|A|$ half-arc Lambek conditions. They are of the form $xy=y'x'$ where $x,x' \in A \times F$ and $y,y'\in F \times P$. Fix $a\in A$, the edge $(a,R(a))$ (resp.\ $(a,L(a))$) appears twice in the totality of half-arc Lambek conditions, always on the left of the condition $xy=y'x'$ (resp.\ right), once in \eqref{eq:ha-left} and once in \eqref{eq:ha-right} and can appear nowhere else. The edge $(L(a),r(a))$ also appears twice, once in \eqref{eq:ha-left} (on the right) and once in $h_{bl}$ where $s(b) =r(a)$ (on the left)
\end{rmks}

\begin{example} \label{ex:C1}
Consider the following polyhedral graph $C_1=(P,A)$ which has been drawn and labelled helpfully to closely resemble its Lambek quadrangle club. Its corresponding convex polyhedron may be formed by deleting the edge $a_1$ for $C_2$ as shown in \eqref{eq:conv-poly} and renumbering.
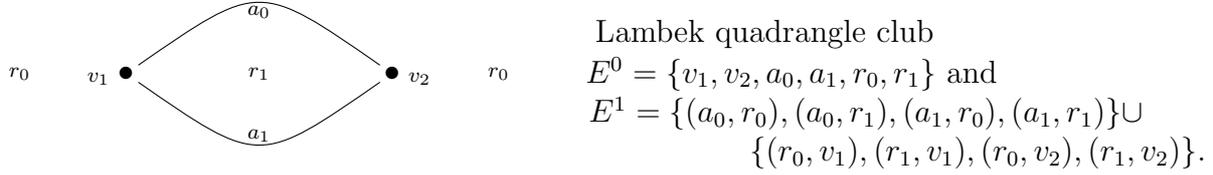
\[
\begin{tikzpicture}[scale=3.5]

\node[inner sep=1pt, circle] (00) at (0,0) {$\bullet$};
\node[inner sep=1pt, circle] (10) at (1,0) {$\bullet$};

\node at (0.5,0) {\tiny $r_1$};
\node at (-0.4,0) {\tiny $r_0$};
\node at (1.4,0) {\tiny $r_0$};

\draw (10) .. controls +(-0.5,+0.35) .. (00) node[pos=0.5,anchor=north,inner sep=1pt, black] {\tiny$a_0$};

\draw (00) .. controls +(0.5,-0.35) .. (10) node[pos=0.5,anchor=south,inner sep=1pt, black] {\tiny$a_1$};

\node[left] at (-0.02,-0.02) {\tiny\color{black} $v_1$};

\node[right] at (1.02,-0.02) {\tiny\color{black} $v_2$};

\node at (2.4,0.15) {Lambek quadrangle club};
\node at (2.5,0) {$E^0 = \{ v_1 , v_2 ,a_0 , a_1 , r_0 , r_1 \}$ and};
\node at (2.78,-0.15){$E^1 = \{ ( a_0 , r_0 ) , ( a_0 , r_1 ) , ( a_1 ,r_0 ) , ( a_1 , r_1) \}\cup$};
\node at (3.2,-0.3) {$\{ (  r_0 , v_1 ) , (  r_1 ,v_1 ) , (r_0 , v_2) , ( r_1 , v_2 ) \} .$};
\end{tikzpicture}
\]

\noindent
Here below we draw the quadrangle club graph $E$, we draw it on on top of polyhedral graph $(P,A)$ so that one can see how it is constructed. The vertex $r_0$, corresponding to the infinite region $r_0$, is repeated to make the picture easier to visualise when drawn on paper. A choice of orientation $R(a_0)=r_0=R(a_1)$ has been made.
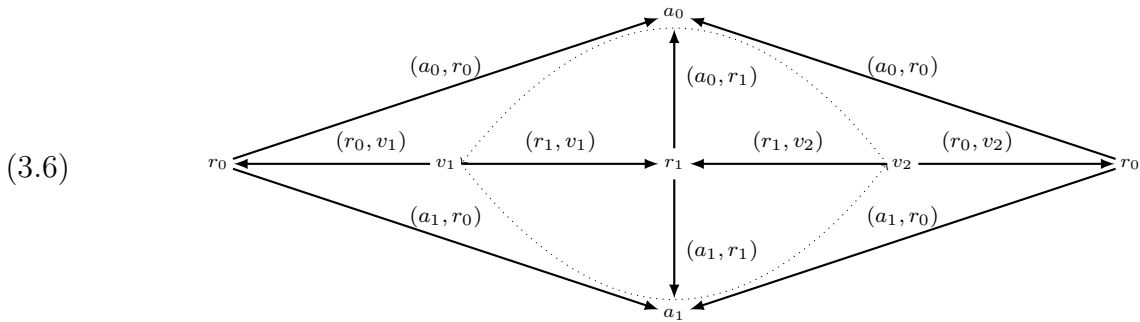
\begin{equation} \label{eq:C2uncoloured}
\begin{array}{c}
\begin{tikzpicture}[->=stealth]

\node[inner sep=1pt,circle] (p0) at (-3,0) {\Tiny $v_1$};
\node[inner sep=1pt,circle] (p1) at (3,0) {\Tiny $v_2$};

\node[inner sep=1pt,circle] (f0) at (0,0) {\Tiny $r_1$};

\node[inner sep=1pt,circle] (lf1) at (-6,0) {\Tiny $r_0$};
\node[inner sep=1pt,circle] (rf1) at (6,0) {\Tiny $r_0$};

\node[inner sep=1pt,circle] (a1) at (0,-2) {\Tiny $a_1$};
\node[inner sep=1pt,circle] (a0) at (0,2) {\Tiny $a_0$};


\draw[dotted] (2.8,0) parabola bend (0,1.8) (-2.8,0);

\draw[dotted] (-2.8,0) parabola bend (0,-1.8) (2.8,0);


\draw[thick,-latex] (p0) -- (f0) node[pos=0.5,above] {\tiny\color{black}$(r_1,v_1)$};

\draw[thick,-latex] (p0) -- (lf1) node[pos=0.3,above] {\tiny\color{black}$(r_0,v_1)$};

\draw[thick,-latex] (p1) -- (f0) node[pos=0.5,above] {\tiny\color{black}$(r_1,v_2)$};

\draw[thick,-latex] (p1) -- (rf1) node[pos=0.3,above] {\tiny\color{black}$(r_0,v_2)$};


\draw[thick,-latex] (lf1) -- (a0) node[pos=0.5,above] {\tiny\color{black}$(a_0,r_0)$};

\draw[thick,-latex] (f0) -- (a0) node[pos=0.6,right] {\tiny\color{black}$(a_0,r_1)$};

\draw[thick,-latex] (rf1) -- (a1) node[pos=0.5,above] {\tiny\color{black}$(a_1,r_0)$};

\draw[thick,-latex] (f0) -- (a1) node[pos=0.6,right] {\tiny\color{black}$(a_1,r_1)$};


\draw[thick,-latex] (rf1) -- (a0) node[pos=0.5,above] {\tiny\color{black}$(a_0,r_0)$};

\draw[thick,-latex] (lf1) -- (a1) node[pos=0.5,above] {\tiny\color{black}$(a_1,r_0)$};

\end{tikzpicture}
\end{array}
\end{equation}

\noindent
There are two pairs of half-arcs $\{ (a_0 , v_1 ), ( a_0 , v_2 ) \}$, $\{ (a_1 , v_2 ) , (a_1,v_1  )\}$ giving rise to four quadrangles $q_{0l}, q_{0r}$, $q_{1l}, q_{1r}$ in $E$. The half-arc Lambek conditions are
\begin{equation} \label{eq:c1-lambek}
\begin{array}{rl}
h_{0l} :=  (a_0,r_0)(r_0,v_1)=(a_0,r_1)(r_1,v_1) & \quad
h_{0r} := (a_0,r_0) (r_0,v_2)=(a_0,r_1)(r_1,v_2) \\
h_{1l} :=  (a_1,r_0) (r_0,v_1)=(a_1,r_1)(r_1,v_1) & \quad
h_{1r} :=  (a_1,r_0) (r_0,v_2)=(a_1,r_1)(r_1,v_2) .
\end{array}
\end{equation}
\end{example}

\begin{rmks} \label{rmk:Imfree}
A choice of orientation was made during the construction of $E$ from $C_1$ in Example~\ref{ex:C1}. A different choice of orientation would result in the same equations being generated in \eqref{eq:c1-lambek} but in a different order. The process discussed in Definition~\ref{dfn:addquads} identifies certain  squares in the directed graph $E$ formed by the quadrangle club, see \eqref{eq:C2uncoloured} above. From this observation one may see that one may move freely between the polyhedral graph and the Lambek quadrangle club.

Note that the identification of squares in a Lambek quadrangle club $E$ allows us to show that it is bipartite: All vertices from $P$ connect to $F$ and all vertices from $F$ connect to $A$.
\end{rmks}

\section{The higher-rank graph $\Lambda_{E,\mathcal{C}}$  built from the Lambek quadrangle club $E$ defined from a polyhedral graph $\mathfrak{P}=(P,A)$} \label{sec:three}

\noindent
In the last section we showed how to associate a Lambek quadrangle club, a directed graph $E$, to a polyhedral graph $\mathfrak{P}=(P,A)$. Then we showed how to colour the natural equations coming from the quadrangles so that $E$ becomes a coloured graph $(E,c)$. In this section we show that this coloured graph has all the requisite properties to become a higher-rank graph $\Lambda_{E,\mathcal{C}}$ where $\mathcal{C}$ are the commuting squares given by the coloured quadrangles.

\subsection{Colouring the Lambek quadrangle club} \label{sec:threepointone}

Here the idea is to colour the edges of Lambek quadrangle club $E$ so that the Lambek half-arc equations \eqref{eq:ha-left}, \eqref{eq:ha-right} colour the squares of $E$ to form a coloured graph with a complete set of squares in the sense of Definition~\ref{dfn:complete-squares}. This will enable us to produce a higher-rank graph using \cite{hrsw} as we had set out to do. 

To  create the necessary complete set of squares in $E$, we colour the edges occurring in the half arc Lambek conditions \eqref{eq:ha-left}  \eqref{eq:ha-right}. Thus creating coloured commuting squares in the sense of \cite{hrsw}. In particular, we must assign edges on opposite sides of the square in \eqref{eq:ha-left}, \eqref{eq:ha-right} the same colour. For certain types of polyhedral graph, this turns out to be particularly easy:

\begin{rmk} \label{rmk:dual}
The \emph{dual} of a planar graph $\mathfrak{P}=(P,A)$ is (naturally) a planar graph $\widehat{\mathfrak{P}}=(\hat{F},\hat{A})$ whose vertices $\{ v_f : f  \in F \}$ are identified with the faces of $\mathfrak{P}$. There is an arc $(f_1,f_2)$ between faces $f_1,f_2$ if there is an arc $a \in A$ such that $L(a)=f_1$, $R(a)=f_2$ (there is also an edge $(f_2,f_1)$ if $L(a)=f_2$, $R(a)=f_1$). The faces $\hat{F} = \{ f_p : p \in P \}$ of $\widehat{\mathfrak{P}}$ are identified with the vertices  of $\mathfrak{P}$. The vertices $\hat{P} = \{ v_f : f \in F \}$ of $\widehat{\mathfrak{P}}$ are identifed with the faces of $\mathfrak{P}$. Since $\mathfrak{P}$ is an Eulerian polygon (polyhedral graph) by Steinitz's Theorem \cite{St} the dual graph is uniquely determined up to isomorphism. 

The Lambek Quadrangle Club $\hat{E} $ of $\widehat{\mathfrak{P}}$ has vertices $\hat{P} \sqcup \hat{A} \sqcup\hat{F}$ and edges
\begin{align*}
\{ ( f_p , v_f ) : (p,f) \in E^1 \} & \text{ or } \\
\{ (L(a),R(a)),f_p) : (R(a),p) \in E^1 \text{ or } (L(a),p) \in E^1 \} & \text{ with the expected range and source maps.}
\end{align*}

\noindent
The graph $E$ and its dual $\hat{E}$ associated to a convex polyhedron $\mathfrak{P}=(P,A)$ will have different numbers of sources ($|P|$ for $E$ and $|F|$ for $\hat{E}$) and sinks ($|A|$ for $E$ and $|\hat{A}|$ for $\hat{E}$). Two examples of polyhedral graphs and their duals drawn on the same picture are to be found in \eqref{eq:conv-poly} and \eqref{eq:c2dual}.
\end{rmk}

\noindent
We thank the contributors to Math.Stackexchange for directing us to a proof which otherwise we could not cite (see \cite{MSt}).

\begin{prop} \label{prop:2-colour}
Let $\mathfrak{P} = (P,A)$ be a polyhedral graph with points $R = \{r_1 , \ldots , r_n \}$ in which every point has even valency. Then there is a function $c : R \to \{ c_1 , c_2 \}$ such that no two adjacent  regions have the same colour.
\end{prop}

\begin{proof}
Since every point in $\mathfrak{P}$ has even valency, the same is true of the dual graph $\widehat{\mathfrak{P}} = (F,\hat{A})$. Hence $\widehat{\mathfrak{P}}$ has no odd cycles and so by \cite[Proposition 1.6.1]{D} it is bipartite. Hence $F = F_1 \sqcup F_2$ and no two faces in $F_1$ (resp.\ $F_2$) are adjacent. Define $c : R \to \{ c_1 , c_2 \}$ by $c ( f ) = c_i$ if and only if $f \in F_i$, $i=1,2$. This completes the proof.
\end{proof}

\begin{example}[Lunar examples] \label{ex:bush}
In Bush \cite{B}, a family of polyhedral graphs $C_n$, $n \ge 1$ are described  where the direct Ma{l}'cev and half-arc Lambek conditions overlap. From \cite[p.\ 440]{B} we have the following: $V=\{ v_1,v_2\}$, $A=\{a_0 , \ldots , a_{n-1} \}$, $F=\{ r_0 , \ldots , r_n \}$. (Hence $|V|=2$, $|A|=n$ and $|F|=n+1$ so $V+F=A+2$). In the diagram given below we show the vertices $\{ v_1, v_2 \}$, half-arcs  $\{ (a_0,v_1), (a_0,v_2)\} , \ldots, \{ (a_{n-1}, v_1) , (a_{n-1} , v_2) \}$, and the regions $r_0 , \ldots , r_n$.

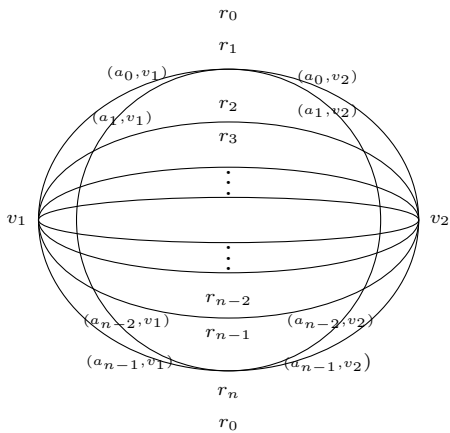
\begin{minipage}[l]{0.35\textwidth}
\[
\begin{tikzpicture}

\draw (0,0) circle (2.5cm);

\draw (0,0) ellipse (2.5cm and 2cm);
\draw (0,0) ellipse (2.5cm and 1.3cm);
\draw (0,0) ellipse (2.5cm and 0.7cm);
\draw (0,0) ellipse (2.5cm and 0.3cm);

\node at (-1.2,1.95) {\Tiny $\scriptscriptstyle (a_0,v_1)$};
\node at (-1.3,-1.9) {\Tiny $\scriptscriptstyle (a_{n-1},v_1)$};
\node at (1.3,1.9) {\Tiny $\scriptscriptstyle (a_0,v_2)$};
\node at (1.3,-1.9) {\Tiny $\scriptscriptstyle (a_{n-1},v_2$)};

\node at (-1.4,1.35) {\Tiny $\scriptscriptstyle (a_1,v_1)$};
\node at (-1.34,-1.35) {\Tiny $\scriptscriptstyle (a_{n-2},v_1)$};
\node at (1.3,1.45) {\Tiny $\scriptscriptstyle (a_1,v_2)$};
\node at (1.34,-1.35) {\Tiny $\scriptscriptstyle (a_{n-2},v_2)$};

\node[left] at (-2.5,0) {\Tiny $v_1$};
\node[right] at (2.5,0) {\Tiny $v_2$};

\node[below] at (0,-2.5) {\Tiny $r_0$};
\node[above] at (0,-2.5) {\Tiny $r_n$};

\node[below] at (0,2.5) {\Tiny $r_1$};
\node[above] at (0,2.5) {\Tiny $r_0$};

\node[below] at (0,-1.3) {\Tiny $r_{n-1}$};
\node[above] at (0,-1.3) {\Tiny $r_{n-2}$};

\node[below] at (0,1.3) {\Tiny $r_3$};
\node[above] at (0,1.3) {\Tiny $r_2$};

\node at (0,0.6) {$\vdots$};
\node at (0,-0.4) {$\vdots$};

\end{tikzpicture}
\]
\end{minipage}%
\begin{minipage}[l]{0.6\textwidth}
\noindent
The  half-arc Lambek lunar conditions for $C_n$, $n \ge 1$ are
\begin{align*}
h_{il} &:= (a_i ,r_i )(r_i ,v_1) =  (a_i,r_{i+1\!\!\!\!\! {\pmod 2}})(r_{i+1\!\!\!\!\! {\pmod 2}} ,v_1)   \ \text{ and } \\
h_{ir} &:=  (a_i , r_i) (r_i ,v_{2} ) = (a_i,r_{i+1\!\!\!\!\! {\pmod 2}})(r_{i+1\!\!\!\!\! {\pmod 2}}, v_2) ,  \\
& 
\text{where } 0\le i\le n-1. 
\end{align*}

\medskip
\hphantom{abd} 
\noindent $C_1$ is the polyhedral graph in Example~\ref{ex:colourex} and $C_2$ is the polyhedral graph in \eqref{eq:conv-poly}. It is easy to see that if $n$ is odd then $C_n$ is $2$-colourable by Proposition~\ref{prop:2-colour} as $v_1 ,v_2$ have even degree. Furthermore, if $n$ is even then it is easy to see that $C_n$ is $3$-colourable using the fact that $C_{n-1}$ is $2$-colourable.
\end{minipage}
\end{example}

\subsubsection*{General colouring algorithm:} \label{sec:coloring-alg}

First check if every vertex in the polyhedral graph has even degree, then by Proposition~\ref{prop:2-colour} we know that the plane can be coloured by two colours, and proceed by the method suggested there. It is known that there is no polynomial-time algorithm for determining if a planar graph is $3$-colourable.
Observe that if a planar polyhedral graph is $2$-colourable, then its dual is not necessarily $2$-colourable, see \eqref{eq:conv-poly}.

Since $\mathfrak{P}= (P,A)$ is planar we know by \cite{AH1,AHK,AH2} that we need at most $4$ colours. Let $R = \{ r_1 , \ldots , r_n \}$ be the regions within $\mathfrak{P}$ and fix a colouring function $c : R \to \{ c_1 , c_2 , c_3 , c_4 \}$ such that for each region/face $f_k$ of $\mathfrak{P}$ such that no two adjacent regions have the same colour.

\begin{rmk}
The choice of colouring function $c$ above gives rise to  some ambiguity about the nature of the resulting coloured graph, and hence the rank of the resulting graph. If two different colouring functions give rise to graphs of different dimension, they will be quasi-isomorphic.

However, if we know for sure that the graph is $2$- or $3$-colourable (see Example~\ref{ex:bush}) then this problem does not arise.
\end{rmk}

\begin{example}
The convex polyhedron shown below may be coloured in two different ways:
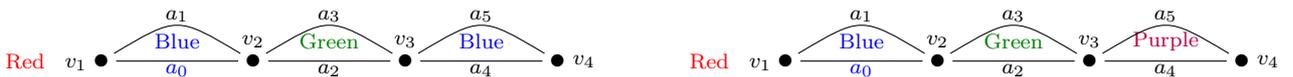
\[
\begin{tikzpicture}[scale=2]
    
\node[inner sep=1pt, circle] (00) at (0,0) {$\bullet$};
\node[inner sep=1pt, circle] (10) at (1,0) {$\bullet$}; 
\node[inner sep=1pt, circle] (20) at (2,0) {$\bullet$}; 
\node[inner sep=1pt, circle] (30) at (3,0) {$\bullet$};
    
\node at (0.5,0.13) {{\color{blue} \tiny Blue}};
\node at (-0.5,0) {{\color{red} \tiny Red}};   
\node at (1.5,0.13) {{\color{green!50!black} \tiny Green}};
\node at (2.5,0.13) {{\color{blue} \tiny Blue}};
    
\draw (00)--(10) node[pos=0.5, anchor=north,inner sep=1pt] {\tiny\color{black}${\color{blue}a_0}$};

\draw (10)--(20) node[pos=0.5, anchor=north,inner sep=1pt] {\tiny\color{black}$a_2$};

\draw (20)--(30) node[pos=0.5, anchor=north,inner sep=1pt] {\tiny\color{black}$a_4$};
 
\draw (10) .. controls +(-0.5,+0.3) .. (00) node[pos=0.5,anchor=south,inner sep=1pt, black] {\tiny$a_1$};

\draw (20) .. controls +(-0.5,+0.3) .. (10) node[pos=0.5,anchor=south,inner sep=1pt, black] {\tiny$a_3$};

\draw (30) .. controls +(-0.5,+0.3) .. (20) node[pos=0.5,anchor=south,inner sep=1pt, black] {\tiny$a_5$};

\node[left] at (-0.02,-0.02) {\tiny\color{black} $v_1$};
\node[above] at (1,0.02) {\tiny\color{black} $v_2$};
\node[above] at (2,0.02) {\tiny\color{black} $v_3$};

\node[right] at (3.03,0) {\tiny\color{black} $v_4$};

\begin{scope}[xshift=4.5cm]

\node[inner sep=1pt, circle] (00) at (0,0) {$\bullet$};
\node[inner sep=1pt, circle] (10) at (1,0) {$\bullet$}; 
\node[inner sep=1pt, circle] (20) at (2,0) {$\bullet$}; 
\node[inner sep=1pt, circle] (30) at (3,0) {$\bullet$};
    
\node at (0.5,0.13) {{\color{blue} \tiny Blue}};
\node at (-0.5,0) {{\color{red} \tiny Red}};   
\node at (1.5,0.13) {{\color{green!50!black} \tiny Green}};
\node at (2.5,0.13) {{\color{purple} \tiny Purple}};
    
\draw (00)--(10) node[pos=0.5, anchor=north,inner sep=1pt] {\tiny\color{black}${\color{blue}a_0}$};

\draw (10)--(20) node[pos=0.5, anchor=north,inner sep=1pt] {\tiny\color{black}$a_2$};

\draw (20)--(30) node[pos=0.5, anchor=north,inner sep=1pt] {\tiny\color{black}$a_4$};
 
\draw (10) .. controls +(-0.5,+0.3) .. (00) node[pos=0.5,anchor=south,inner sep=1pt, black] {\tiny$a_1$};

\draw (20) .. controls +(-0.5,+0.3) .. (10) node[pos=0.5,anchor=south,inner sep=1pt, black] {\tiny$a_3$};

\draw (30) .. controls +(-0.5,+0.3) .. (20) node[pos=0.5,anchor=south,inner sep=1pt, black] {\tiny$a_5$};

\node[left] at (-0.02,-0.02) {\tiny\color{black} $v_1$};
\node[above] at (1,0.02) {\tiny\color{black} $v_2$};
\node[above] at (2,0.02) {\tiny\color{black} $v_3$};

\node[right] at (3.03,0) {\tiny\color{black} $v_4$};

\end{scope}
     
\end{tikzpicture}
\]

\noindent
By Theorem~\ref{thm:polygiveshrg} the left polyhedron gives rise to a rank-$3$ graph $\Lambda_1$ and the right polyhedron gives rise to a graph $\Lambda_2$ of rank $4$, which is quasi-isomorphic to $\Lambda_1$.
\end{example}

\begin{dfn}[Colouring quadrangles] \label{dfn:colouring}
Let $\mathfrak{P}=(P,A)$ be a polyhedral graph, and $E$ the associated Lambek quadrangle club. Let $Q = \{ q_{ar} , q_{al} : a \in A \}$ be the collection of quadrangles.
\begin{itemize} 
\item[$q_{a l}$:] \label{qal} In the quadrangle $q_{a l}$ we associate the colour $c(R(a))$ associated to the region $R(a)$ to the edges $(L(a),r(a))$, $(a,R(a))$ (see \eqref{eq:ha-left} and \eqref{eq:ha-right}). By definition opposite edges $(R(a),r(a))$, $(a,(L(a))$ are assigned the (different) colour $c(L(a))$ of the region $L(a)$.
\item[$q_{a r}$:] \label{qar}  In the quadrangle $q_{a r}$ we associate the colour $c(R(a))$ associated to the region $R(a)$ to the edges $(L(a),s(a))$, $(a,R(a))$ (see \eqref{eq:ha-left} and \eqref{eq:ha-right}). By definition opposite edges $(R(a),s(a))$, $(a,(L(a))$ are assigned the (different) colour $c(L(a))$ of the region $L(a)$. 
\end{itemize}
\end{dfn}

\begin{example} \label{ex:colourex}
Let $\mathfrak{P}=C_1$ be the polyhedral graph shown below drawn with solid lines. We have also drawn the dual graph (dotted) $\widehat{\mathfrak{P}}= ( \{r_0,r_1\} , \{ (r_0,r_1),(r_1,r_0)\})$ (see Remark~\ref{rmk:dual}). Then the quadrangle club $E$  is formed as shown below.
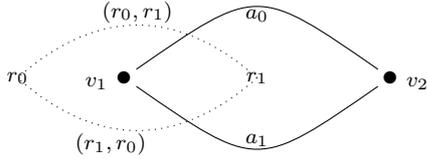
\begin{equation} \label{eq:c2dual}
\begin{array}{l}
\begin{tikzpicture}[scale=3.5]

\node[inner sep=1pt, circle] (00) at (0,0) {$\bullet$};
\node[inner sep=1pt, circle] (10) at (1,0) {$\bullet$};

\node at (0.5,0) {\tiny $r_1$};
\node at (-0.4,0) {\tiny $r_0$};

\draw[dotted] (0.5,0) parabola bend (0.05,0.2) (-0.4,0);
\draw[dotted] (0.5,0) parabola bend (0.05,-0.2) (-0.4,0);
\node at (0.05,0.25) {\tiny $(r_0,r_1)$};
\node at (-0.05,-0.25) {\tiny $(r_1,r_0)$};

\draw (10) .. controls +(-0.5,+0.35) .. (00) node[pos=0.5,anchor=north,inner sep=1pt, black] {\tiny$a_0$};

\draw (00) .. controls +(0.5,-0.35) .. (10) node[pos=0.5,anchor=south,inner sep=1pt, black] {\tiny$a_1$};

\node[left] at (-0.02,-0.02) {\tiny\color{black} $v_1$};

\node[right] at (1.02,-0.02) {\tiny\color{black} $v_2$};

\node at (2.4,0.15) {Lambek quadrangle club};
\node at (2.5,0.0) {$E^0 = \{ v_1 , v_2 ,a_0 , a_1 , r_0 , r_1 \}$ and};
\node at (2.78,-0.15){$E^1 = \{ ( a_0 , r_0 ) , ( a_0 , r_1 ) , ( a_1 ,r_0 ) , ( a_1 , r_1) \}\cup$};
\node at (3.2,-0.3) {$\{ (  r_0 , v_1 ) , (  r_1 ,v_1 ) , (r_0 , v_2) , ( r_1 , v_2 ) \} .$};

\end{tikzpicture}
\end{array}
\end{equation}

\noindent
Since every vertex has even degree we may associate the colour blue to the region $r_0$ and the colour red to the region $r_1$ and an orientation such that $R(a_0)=r_0$ and $R(a_1)=r_0$. There are two half arc pairs $\{ (a_0 , v_1 ) , (a_0 , v_2 ) \}$ and $\{ ( a_1 , v_2 ) , (a_1 , v_1 ) \}$ giving rise to four quadrangles
\begin{align*}
q_{0l}  = [(r_1,v_1), (r_0,v_1),(a_0,r_1),(a_0,r_0)], \ & \quad q_{1l}  = [(r_1,v_2), (r_0,v_2),(a_0,r_1),(a_0,r_0)],  \\ 
q_{0r}  = [(r_1,v_1), (r_0,v_1),(a_1,r_1),(a_1,r_0)], \ & \quad  q_{1r} =  [(r_1,v_2), (r_0,v_2),(a_1,r_1),(a_1,r_0)] . 
\end{align*} 

\noindent
In quadrangle $q_{0l}$  we colour the edges $(a_0,r_0)$, $(r_1,v_1 )$ blue and the edges $(r_0,v_1)$, $(a_0,r_1)$ red and so on. Here below we draw the $2$-coloured Lambek quadrangle club graph,  we draw it on on top of polyhedral graph $\mathfrak{P}=(P,A)$ so that one can see how it is constructed. 
\begin{equation} \label{eq:c1-coloured}
\begin{array}{l}
\begin{tikzpicture}[->=stealth,scale=0.7]

\node[inner sep=1pt,circle] (p0) at (-3,0) {\Tiny $v_1$};
\node[inner sep=1pt,circle] (p1) at (3,0) {\Tiny $v_2$};

\node[inner sep=1pt,circle] (f0) at (0,0) {\Tiny $r_1$};

\node[inner sep=1pt,circle] (lf1) at (-6,0) {\Tiny $r_0$};
\node[inner sep=1pt,circle] (rf1) at (6,0) {\Tiny $r_0$};

\node[inner sep=1pt,circle] (a1) at (0,-2) {\Tiny $a_1$};
\node[inner sep=1pt,circle] (a0) at (0,2) {\Tiny $a_0$};


\draw[dotted] (2.8,0) parabola bend (0,1.8) (-2.8,0);

\draw[dotted] (-2.8,0) parabola bend (0,-1.8) (2.8,0);


\draw[thick,blue,-latex] (p0) -- (f0) node[pos=0.5,above] {\tiny\color{black}$(r_1,v_1)$};

\draw[dashed,red,thick,-latex] (p0) -- (lf1) node[pos=0.3,above] {\tiny\color{black}$(r_0,v_1)$};

\draw[thick,blue,-latex] (p1) -- (f0) node[pos=0.5,above] {\tiny\color{black}$(r_1,v_2)$};

\draw[red,dashed,thick,-latex] (p1) -- (rf1) node[pos=0.3,above] {\tiny\color{black}$(r_0,v_2)$};


\draw[thick,blue,-latex] (lf1) -- (a0) node[pos=0.5,above] {\tiny\color{black}$(a_0,r_0)$};

\draw[red,dashed,thick,-latex] (f0) -- (a0) node[pos=0.6,right] {\tiny\color{black}$(a_0,r_1)$};

\draw[blue,thick,-latex] (rf1) -- (a1) node[pos=0.5,above] {\tiny\color{black}$(a_1,r_0)$};

\draw[red,dashed,thick,-latex] (f0) -- (a1) node[pos=0.6,right] {\tiny\color{black}$(a_1,r_1)$};


\draw[thick,blue,-latex] (rf1) -- (a0) node[pos=0.5,above] {\tiny\color{black}$(a_0,r_0)$};

\draw[thick,blue,-latex] (lf1) -- (a1) node[pos=0.5,above] {\tiny\color{black}$(a_1,r_0)$};

\begin{scope}[xshift=4.7cm]
\node at (5.2,2) {Half-arc Lambek conditions:};
\node at (7,0) {$\begin{array}{l}
h_{0l} := {\color{blue}(a_0,r_0)}{\color{red}(r_0,v_1)} ={\color{red}(a_0,r_1)} {\color{blue}(r_1,v_1)} ,\\ 
h_{1l} := {\color{blue}(a_0,r_0)}{\color{red}(r_0,v_2)} = {\color{red}(a_0,r_1)}{\color{blue}(r_1,v_2)} , \\
h_{0r} := {\color{blue}(a_1,r_0)}{\color{red}(r_0,v_1)} ={\color{red}(a_1,r_1)}{\color{blue}(r_1,v_1)} , \\
 h_{1r} :=  {\color{blue}(a_1,r_0)}{\color{red}(r_0,v_2)} = {\color{red}(a_1,r_1)}{\color{blue}(r_1,v_2)} .
\end{array}$};
\end{scope}
\end{tikzpicture}
\end{array}
\end{equation}
\end{example}

\subsection{The coloured Lambek quadrangle club $(E,c)$}

Following section~\ref{sec:coloring-alg}, we know that $(E,c)$ is a $k$-coloured graph, $2 \le k \le 4$. Let $Q = \{ q_{ar} , q_{al} : a \in A \}$ be the collection of quadrangles in $E$, with associated bi-coloured commuting squares given by the half-arc 
Lambek conditions.
Each quadrangle gives rise to a bi-coloured square in the coloured graph $(E,c)$: For instance, fix $a \in A$ and consider the quadrangle
\[
q_{al} = \big( (a,R(a)) , (R(a),r(a)) ,  (a,L(a)) , (L(a),r(a)) \big)
\]

\noindent
where $(a,R(a)),(L(a),r(a))$ have colour $c_i$ and $(a,L(a)), (R(a),r(a))$ have colour $c_j$. Then following section~\ref{sec:CG} and consulting \eqref{eq:half-arcs} we construct a square $\phi_{al}$ in $(E,c)$ by defining $\phi_{al} : ( E_{k,(\varepsilon_i+\varepsilon_j)} , c_E) \to (E,c)$ by
\begin{align*}
\phi_{al} (0) &= r(a) , \ \phi_{al} ( \varepsilon_i ) = R(a) , \ \phi_{al} ( \varepsilon_j) = L(a) \text{ and } \phi_{al} ( \varepsilon_i  + \varepsilon_j ) = s(a) , \\
\phi_{al} ( f^0_i ) &= (a,R(a)), \ \phi_{al} ( f^{\varepsilon_i}_j ) = (r(a),L(a)), \ \phi_{al} ( f_0^j ) = (a,L(a)) , \ \phi_{al} ( f^{\varepsilon_j}_i ) = (r(a),R(a)) ,
\end{align*}

\noindent
and similarly for a square $\phi_{ar}$ for the quadrangle $q_{ar}$. The the set $\Cc$ of $2 |A|$ commuting squares in $(E,c)$ is then
\begin{equation} \label{eq:cs}
\Cc = \{ \phi_{al} , \phi_{ar} : a \in A \} .
\end{equation}

\noindent
Recall from Definitions~\ref{dfn:complete-squares}, a collection of bi-coloured squares $\Cc$ in $(E,c)$ is complete if for each $i \neq j \leq k$ and $c_ic_j$-coloured path $fg \in E^2$ there is a unique $c_jc_i$-coloured path $ g'f'$ with the same range and source such that $(fg, g'f') \in \Cc$.

\begin{thm} \label{thm:polygiveshrg}
Let $\mathfrak{P}=(P,A)$ be a convex polyhedron and $E$ its associated Lambek quadrangle club which has been $k$-coloured, $2 \le k \le 4$, according to Definition~\ref{dfn:colouring}. Then the $2|A|$ commuting squares $\mathcal{C}$ \eqref{eq:cs} given by the Lambek half-arc equations \eqref{eq:ha-left} and \eqref{eq:ha-right} are  complete.  Hence $(E,c)$ with commuting squares $\mathcal{C}$ gives rise to a connected, singly connected, locally convex rank-$k$ graph 
$\Lambda_{(E,\Cc)}$.
\end{thm}

\begin{proof}
First we note that since there are no paths of length three or more in $E$, so the associativity condition of \cite[\S 3]{hrsw} does not apply. 
To check completeness there are only two cases to consider as paths of length $2$ in $E$ have only one possible form. For the case $fg= (a,R(a)) (R(a),r(a)) \in E^2$ for some $a \in A$ with colours $c_i c_j$, $i \neq j$. The range and source tell us that this path forms part of the Lambek half-arc equation $q_{al} = (a,R(a)) (R(a),r(a)) = (a,L(a)) (L(a),r(a))$. By Definition~\ref{dfn:colouring} the path $g'f'= (a,L(a)) (L(a),r(a))$ is $c_jc_i$-coloured and  is the unique path such that $(fg,g'f') \in \Cc$. For the case $fg = (a,R(a))(R(a),s(a))$ for some $a \in A$ with colours $c_i c_j$, $i \neq j$, a similar argument applies. Finally, applying \cite[Theorem 4.5]{hrsw} completes the proof. By the standing assumption $\mathfrak{P}$ is connected and so we can see from \eqref{eq:half-arcs} that the three different types $P,A,F$ comprising the vertices of $E$ are mutually connected. Hence the $1$-skeleton $\Sk_{\Lambda_{(E,\Cc)}}$, is connected. That $\Lambda_{(E,\Cc)}$ is singly connected follows from the statement about paths of length two in $E = \Sk_{\Lambda_{(E,\Cc)}}$ given above.
Since the rank-$2$ graph is made up of bi-coloured quadrangles in which the local convexity condition holds, it follows that the graph itself is locally convex.
\end{proof}

\begin{rmk} \label{rmk:opposite}
One may carry out the above construction but instead turn the arrows around in a quadrangle (see \eqref{eq:half-arcs}). Everything works through, except that the half-arc equations \eqref{eq:ha-left} and \eqref{eq:ha-left} must be reversed to account for the new quadrangles. The resulting rank-$k$ graph $\Lambda_{(E^{opp},\Cc^{opp})}$ will be the opposite of the one constructed in Theorem~\ref{thm:polygiveshrg}.

Theorem~\ref{thm:polygiveshrg} also tells us that the quadrangle club $\hat{E}$ of the dual graph of a convex polyhedron $\mathfrak{P}=(P,F)$ will have a different number, $\hat{\mathcal{C}} = 2|\hat{A}|$ of commuting squares compared to that of $E$. Combined with Remark~\ref{rmk:dual} this indicates that the relationship (if any) between the higher-rank graphs associated to $E$ and $\hat{E}$ is complicated.
\end{rmk}

\section{Computing the Fundamental Group of $\Lambda_{(E,\Cc)}$} \label{sec:five}

\noindent
We claim that each $\Lambda_{E,\mathcal{C}}$ is a tree. In order to show that each $\Lambda_{E,\mathcal{C}}$ is a tree, we must compute its findamentsl group. To do this we plan to apply the results of \cite{KPW}. In order to use these results, we must first chose a maximal spanning tree for the planar $1$-skeleton $\Sk_{\Lambda_{(E,\Cc)}}=E$  of $\Lambda_{(E,\Cc)}$. We do this as follows: Let $\mathfrak{P}=(P,A)$ be a polyhedral graph
where $P=\{ v_1 , \ldots , v_{m} \}$,  $R= \{ r_1 , \ldots , r_{n} \}$, and $A= \{ a_1 \ldots  , a_{m+n-2} \}$,  for $m,n \ge 2$. Let $E$ its Lambek quadrangle club graph which has $N=2 |A|+2$ vertices. Fix a colouring $c : E^1 \to \{ c_1 , \ldots , c_k \}$, $2 \le k \le 4$.  Let $Q = \{ q_{ar} , q_{al} : a \in A \}$ be the collection of quadrangles, with $2 |A|$ associated commuting squares $\Cc$. Let $\Lambda_{(E,\Cc)}$ be the associated higher-rank graph. Recall, in a graph with $n$ vertices, a spanning tree will have $n-1$ edges. We use a variant of the usual maximal spanning tree algorithm which is greedy for edges on the left side of each Lambek half-arc equation.

\subsection{Left-greedy maximal spanning tree algorithm} \label{sec:greedy}
\begin{enumerate}[label={\bf \arabic*.}]
\item  Assign weights to the edges of $E$ in such a way that edges of the form $(a,f)$ where $f=R(a)$ have weight $2$, edges of the form $(a,f)$ some $a \in A$ where $f=R(a)$ have weight $1$. Edges of the form $(r,v)$ where $v=r(a)$ some $a \in A$ where $r=R(a)$ have weight $2$ otherwise $(r,v)$ has weight $1$.  Order the edges in $E^1$ with those of weight $2$ first and those of weight $1$ next. Let $T$ be the set of edges comprising the maximum weight spanning tree. Set $T = \emptyset$.
\item Add the first edge, in $E^1$ with weight $2$ to $T$.
\item Add the next edge with weight $2$ to $T$ if and only if it does not form a cycle in $T$. If there are no remaining edges of weight $2$ add the first edge of weight $1$ which does not form a (n undirected) cycle in $T$.
\item If $T$ has $N-1$ edges stop and output $T$. Otherwise go to step {\bf 3}.
\end{enumerate}

\begin{prop} With notation as above then $T$ is a maximal spanning tree for $\Sk_{\Lambda_{(E,\Cc)}}=E$ with maximum weight.
\end{prop}

\begin{proof}
See, for instance, \cite[Theorem 10]{Bo}.
\end{proof}

\begin{example} \label{ex:treex}
The picture below shows the above algorithm applied to the polyhedral graph $C_1$ with Lambek half-arc conditions \eqref{eq:c1-lambek}. Recall $R(a_0)=r_0=R(a_1)$. Give weight $2$ to edges $(a_0,r_0),(a_1,r_0),(r_0,v_0),(r_0,v_1)$, then edges $(a_0,r_1),(a_1,r_1),(r_1,v_0),(r_1,v_1)$ have weight $1$. These edges are then put together in the order suggested. Note $N = |E^0| = 2|A| +2 = 6$.

First pass $(a_0,r_0)$ accepted. Second pass $(a_1,r_0)$ accepted. Third pass $(r_0,v_0)$ accepted. Fourth  pass $(r_0,v_1)$ accepted. Fourth pass edges of weight $2$ exhausted, $(a_0, r_1)$ accepted, edge count $5$ reached algorithm terminates giving tree $T$ shown below to the right (recall vertex $r_0$ was repeated for convenience in \eqref{eq:c1-coloured} as it is here).
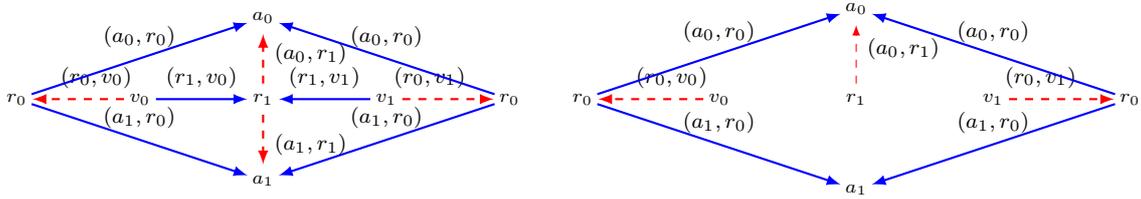
\[
\begin{tikzpicture}[->=stealth,scale=0.6]

\node[inner sep=1pt,circle] (p0) at (-3,0) {\Tiny $v_0$};
\node[inner sep=1pt,circle] (p1) at (3,0) {\Tiny $v_1$};

\node[inner sep=1pt,circle] (f0) at (0,0) {\Tiny $r_1$};

\node[inner sep=1pt,circle] (lf1) at (-6,0) {\Tiny $r_0$};
\node[inner sep=1pt,circle] (rf1) at (6,0) {\Tiny $r_0$};

\node[inner sep=1pt,circle] (a1) at (0,-2) {\Tiny $a_1$};
\node[inner sep=1pt,circle] (a0) at (0,2) {\Tiny $a_0$};

\draw[dashed,red,thick,-latex] (p0) -- (lf1) node[pos=0.3,above] {\tiny\color{black}$(r_0,v_0)$};

\draw[red,dashed,thick,-latex] (p1) -- (rf1) node[pos=0.3,above] {\tiny\color{black}$(r_0,v_1)$};


\draw[thick,blue,-latex] (lf1) -- (a0) node[pos=0.5,above] {\tiny\color{black}$(a_0,r_0)$};

\draw[red,dashed,-latex] (f0) -- (a0) node[pos=0.6,right] {\tiny\color{black}$(a_0,r_1)$};


\draw[thick,blue,-latex] (lf1) -- (a1) node[pos=0.5,above] {\tiny\color{black}$(a_1,r_0)$};

\draw[thick,blue,-latex] (rf1) -- (a0) node[pos=0.5,above] {\tiny\color{black}$(a_0,r_0)$};

\draw[blue,thick,-latex] (rf1) -- (a1) node[pos=0.5,above] {\tiny\color{black}$(a_1,r_0)$};

\begin{scope}[xshift=-13cm,scale=0.9]
\node[inner sep=1pt,circle] (p0) at (-3,0) {\Tiny $v_0$};
\node[inner sep=1pt,circle] (p1) at (3,0) {\Tiny $v_1$};

\node[inner sep=1pt,circle] (f0) at (0,0) {\Tiny $r_1$};

\node[inner sep=1pt,circle] (lf1) at (-6,0) {\Tiny $r_0$};
\node[inner sep=1pt,circle] (rf1) at (6,0) {\Tiny $r_0$};

\node[inner sep=1pt,circle] (a1) at (0,-2) {\Tiny $a_1$};
\node[inner sep=1pt,circle] (a0) at (0,2) {\Tiny $a_0$};



\draw[thick,blue,-latex] (p0) -- (f0) node[pos=0.5,above] {\tiny\color{black}$(r_1,v_0)$};

\draw[dashed,red,thick,-latex] (p0) -- (lf1) node[pos=0.3,above] {\tiny\color{black}$(r_0,v_0)$};

\draw[thick,blue,-latex] (p1) -- (f0) node[pos=0.5,above] {\tiny\color{black}$(r_1,v_1)$};

\draw[red,dashed,thick,-latex] (p1) -- (rf1) node[pos=0.3,above] {\tiny\color{black}$(r_0,v_1)$};


\draw[thick,blue,-latex] (lf1) -- (a0) node[pos=0.5,above] {\tiny\color{black}$(a_0,r_0)$};

\draw[red,dashed,thick,-latex] (f0) -- (a0) node[pos=0.6,right] {\tiny\color{black}$(a_0,r_1)$};

\draw[blue,thick,-latex] (rf1) -- (a1) node[pos=0.5,above] {\tiny\color{black}$(a_1,r_0)$};

\draw[red,dashed,thick,-latex] (f0) -- (a1) node[pos=0.6,right] {\tiny\color{black}$(a_1,r_1)$};


\draw[thick,blue,-latex] (rf1) -- (a0) node[pos=0.5,above] {\tiny\color{black}$(a_0,r_0)$};

\draw[thick,blue,-latex] (lf1) -- (a1) node[pos=0.5,above] {\tiny\color{black}$(a_1,r_0)$};
\end{scope}

\end{tikzpicture}
\]

\noindent
Now we may use \cite[Theorem 5.1]{KPW} to compute the fundamental group.  By \cite[Theorem 5.3]{KPW} the group is generated by $E^1$ (that is, $8$ generators) subject to $4$ relations given by the Lambek half-arc conditions (see below) and setting the $5$ generators (marked blue below) corresponding to the maximal weight maximal spanning tree to be the identity, leaving $3$ generators unknown.
\begin{align*}
h_{0l} := {\color{blue} (a_0,r_0)} {\color{blue} (r_0,v_0)} = {\color{blue}(a_0,r_1)}{(r_1,v_0)}& \quad
h_{0r} := {\color{blue}  (a_0,r_0)}{\color{blue} (r_0,v_1)} = {\color{blue}(a_0,r_1)}{(r_1,v_1)} \\
h_{1l} := {\color{blue} (a_1,r_0)}{\color{blue} (r_0,v_0)}= {(a_1,r_1)}{(r_1,v_0)} & \quad
h_{1r} :=  {\color{blue}(a_1,r_0)} {\color{blue}(r_0,v_1)} = {(a_1,r_1)}{(r_1,v_1)}
\end{align*}

\noindent
A short calculation reveals that all the generators are  equal to the identity and so the group is trivial. 
\end{example}

\begin{rmk}[General case] \label{rmk:summary}
In the example above there are $2 |A|$ quadrangles, each half-arc equation has $4$ variables, each of which is repeated once, so we have $4 |A|$ variables. The maximal weight maximal spanning tree removes $N=(2|A|+2)-1$ variables with a bias to those variables on the left-hand side of \eqref{eq:ha-left} and \eqref{eq:ha-right}, leaving $2|A|$ equations in $2|A|-1$ unknowns.
\end{rmk}

\begin{thm} \label{thm:fgistriv}
Let $\mathfrak{P}=(P,A)$ be a convex polyhedron and $E$ its associated quadrangle club which has been $k$-coloured, $2 \le k \le 4$, according to Definition~\ref{dfn:colouring}. Let $\Lambda_{(E,\Cc)}$ be the associated rank-$k$ graph. Then $\Lambda_{(E,\Cc)}$ is planar and the fundamental group of $\Lambda_{(E,\Cc)}$ is trivial, and hence it is a rank-$k$ tree.
\end{thm}

\begin{proof}
That $\Lambda_{(E,\Cc)}$ is planar follows from the fact that $E$ is planar; it only has undirected cycles of length four and so cannot contain $K_{3,3}$ or $K_5$.

We compute the fundamental group of $\Lambda_{(E,\Cc)}$ using the method of \cite[Theorem 5.1]{KPW}. This result states that the fundamental group $\pi ( \Lambda_{(E,\Cc)})$ is generated by $E^1 \backslash T^1$ subject to the relations \eqref{eq:ha-left} and \eqref{eq:ha-right} with the generators in $T^1$ set to the identity. Following Remark~\ref{rmk:summary} this leaves $2|A|$ equations in $2|A|-1$ unknowns. We claim that at least one equation has three generators set to the identity.
Suppose, for contradiction otherwise, then $2|A|$ generators will have been set to identity at the first step, which contradicts the size of the spanning tree. So another generator is trivial, and the corresponding relation removed and it's paired variable set to the identity. Leaving
$2 |A| -1$ equations in $2|A|-2$ unknowns. Repeating this process we end up with all generators trivial, hence the fundamental group is trivial.
\end{proof}

\begin{dfn}
Let $\mathcal{L}_{\mathfrak{P}}$ denote the collection of rank-$k$ graphs $\Lambda_{(E,\Cc)}$ ($2 \le k \le 4$) which are constructed in this way. We call them Lambek trees.
\end{dfn}

\begin{rmks}
We must quickly remark that we are not claiming that $\mathcal{L}_{\mathfrak{P}}$  describes all higher-rank trees. Any subgraph of a tree constructed using the methods described here is likely to be a higher-rank tree but not be a member of $\mathcal{L}_{\mathfrak{P}}$. Indeed, consider the examples given above: Example~\ref{ex:treex} consists of four leaves $q_{0l}, q_{1l}$ suspended from $v_1$ and $q_{0r}, q_{1r}$ suspended from $v_{2}$ they are glued together along $(a_0,r_0)$ and $(a_1,r_0)$ respectively. 
None of the $\Omega_{k,m}$ examples 
are elements of $\mathcal{L}_{\mathfrak{P}}$.
\end{rmks}

\begin{cor} \label{cor:pthrg-embedd}
Let $\Lambda_{(E,\Cc)} \in \mathcal{L}_{\mathfrak{P}}$, then $\Lambda_{(E,\Cc)}$ embeds in faithfully in its fundamental groupoid.
\end{cor}

\begin{proof}
Define the map $c: \Lambda_{(E,\Cc)} \to \ZZ$ by $c ( a , f ) = 1$ for all $(a,f) \in \Sk_{\Lambda_{(E,\Cc)}}^1=E^1$ and $c ( f , p ) = -1$ for all $(f,p) \in  \Sk_{\Lambda_{(E,\Cc)}}^1=E^1$. One checks that $c$ extends to a functor $c: \Lambda_{(E,\Cc)} \to \ZZ$. Furthermore, since $\Lambda_{(E,\Cc)}$ is singly connected it is easy to see that $c$ is essential. The result then follows by \cite[Proposition~3.12]{bkps}.
\end{proof}

\begin{rmk}
Following Remark~\ref{rmk:opposite0} and \ref{rmk:opposite} and Corollary~\ref{cor:pthrg-embedd} we have $\Lambda_{(E,\Cc)}^{opp} = \Lambda_{(E^{opp},\Cc^{opp})}$ is a rank-$k$ tree for $2 \le k \le 4$ and embeds faithfully in its fundamental groupoid by \cite[Proposition~3.12]{bkps}. It $\Lambda_{(E,\Cc)}$ is planar, then so is
$\Lambda_{(E,\Cc)}^{opp}$.
\end{rmk}

\noindent
Finally, we seek an analog of the formula $| T^1 | - | T^0| +1 =0$ for $T$ a finite rank-$1$ tree. We can get such a formula for members of $\mathcal{L}_{\mathfrak{P}}$

\begin{prop}  \label{prp:Lambek_Euler}
Let $\Lambda_{(E,c)} \in \mathcal{L}_{\mathfrak{P}}$ then  $E = \Sk_{\Lambda_{E,\mathcal{C}}}$ satisfies $|E^1 | - 2 | E^0 |  +4 =0$.
\end{prop}

\begin{proof}
By Remark~\ref{rmk:summary} we have $| E^1 | =4 |A|$ where $\mathfrak{P} = (P,A)$ is the convex polyhedron giving rise to $E$. Furthermore $|E^0| = |P|+|F|+|A|$ by definition of $E^0$. Since $E$ is planar we have $|E^0| = 2|A|+2$, by Euler's formula, and the result follows.
\end{proof}

\begin{lem} \label{lem:acyclic}
Let $\Lambda_{(E,c)} \in \mathcal{L}_{\mathfrak{P}}$ then $\Lambda_{(E,c)}$ is acyclic, that is $H_i ( \Lambda_{(E,c)}  ) = 0$ for $i \ge 1$.
\end{lem}

\begin{proof}
From \cite[Proposition 6.2]{KPW} we have $H_1 ( \Lambda ) = \operatorname{Ab} ( \pi_1 ( \Lambda_{(E,c)} ) ) = 0$. Since there are no paths in the $1$-skeleton $\Sk_{\Lambda_{(E,c)}}$ of length greater than two it follows that the set of $r$-cubes $Q_r ( \Lambda_{(E,c)} ) = \emptyset$ for $r > 2$ (see \cite[Definition 2.10]{KPS3}) and so $H_r ( \Lambda_{(E,c)} )$ for $r > 2$. Since $H_1( \Lambda_{(E,c)}) =0$ we have $\operatorname{ker} (\partial_2)=Q_2( \Lambda_{(E,c)} )$, since $H_3 ( \Lambda_{(E,c)} ) = 0$ we have $\operatorname{Im} ( \partial_3) = Q_2 ( \Lambda_{(E,c)} )$ and so $H_2 ( \Lambda_{(E,c)} ) = 0$. It follows that $H_i ( \Lambda_{(E,c)} ) = 0$ for $i \ge 1$.
\end{proof}



\section{Differences from one-dimensional trees}

\noindent
The purpose of this section is to show that, since higher-rank graphs are not free categories, then it is possible to have a  higher-rank tree with properties different to a rank-$1$ tree.

\vspace{3mm}

\textbf{Fact 1.} A rank-$1$ tree is bipartite with no cycles, so by the $4$-colour theorem it is planar.
\vspace{3mm}

\noindent
The same is not true of higher-rank trees.

\begin{example}
The hypercube graph $Q_4$ shown in \eqref{eq:q4} below to the left may be coloured with $4$ colours so that it becomes a locally convex rank-$4$ graph. Though there are paths of length four (i.e. $Q_4^4 \neq \emptyset$), the associativity condition from \cite{hrsw} is trivially satisfied). A short calculation shows that it is a tree. It is well-known (cf.\ \cite{Wikipedia2024}) that it is not planar: The subgraph of its $1$-skeleton shown below to the right in \eqref{eq:q4} is homotopic to $K_{3,3}$. In three and four dimensions edges of dimension $3$ are coloured green and drawn dotted, edges of dimension $4$ are drawn are coloured purple and drawn dash-dotted. 
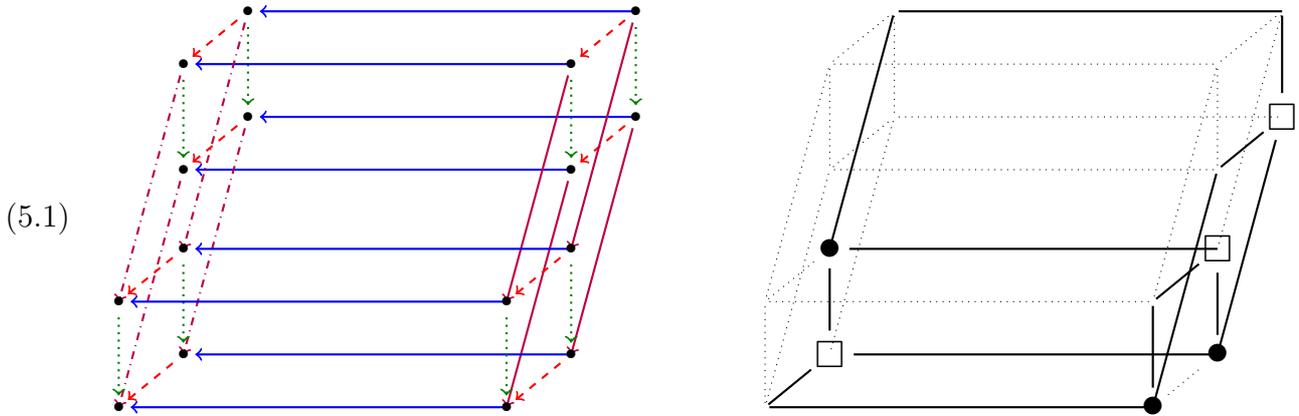
\begin{equation} \label{eq:q4}
\begin{array}{l}
\begin{tikzpicture}[yscale=0.7,xscale=0.85]

\node[inner sep=1pt,circle] (00) at (0,0) {\tiny $\bullet$};
\node[inner sep=1pt,circle] (11) at (1,1) {\tiny $\bullet$};
\node[inner sep=1pt,circle] (02) at (0,2) {\tiny $\bullet$};
\node[inner sep=1pt,circle] (13) at (1,3) {\tiny $\bullet$};

\draw[<-,green!50!black,dotted,thick] (00) -- (02); 
\draw[<-,green!50!black,dotted,thick] (11) -- (13); 
\draw[<-,red,dashed,thick] (00) -- (11);
\draw[<-,red,dashed,thick] (02) -- (13);

\draw[<-,blue,thick] (00) -- (6,0);
\draw[<-,blue,thick] (13) -- (7,3);
\draw[<-,blue,thick] (11) -- (7,1);
\draw[<-,blue,thick] (02) -- (6,2);

\draw[->,purple,dash dot,thick] (00) -- (-1,-4.5);
\draw[->,purple,dash dot,thick] (11) -- (0,-3.5);
\draw[->,purple,dash dot,thick] (02) -- (-1,-2.5);
\draw[->,purple,dash dot,thick] (13) -- (0,-1.5);

\begin{scope}[xshift=6cm]

\node[inner sep=1pt,circle] (00) at (0,0) {\tiny $\bullet$};
\node[inner sep=1pt,circle] (11) at (1,1) {\tiny $\bullet$};
\node[inner sep=1pt,circle] (02) at (0,2) {\tiny $\bullet$};
\node[inner sep=1pt,circle] (13) at (1,3) {\tiny $\bullet$};

\draw[<-,green!50!black,dotted,thick] (00) -- (02); 
\draw[<-,green!50!black,dotted,thick] (11) -- (13); 
\draw[<-,red,dashed,thick] (00) -- (11);
\draw[<-,red,dashed,thick] (02) -- (13);

\draw[->,purple,dash dot, thick] (00) -- (-1,-4.5);
\draw[->,purple,dash dot, thick] (11) -- (0,-3.5);
\draw[->,purple,dash dot, thick] (02) -- (-1,-2.5);
\draw[->,purple,dash dot,thick] (13) -- (0,-1.5);

\end{scope}

\begin{scope}[yshift=-4.5cm,xshift=-1cm]

\node[inner sep=1pt,circle] (00) at (0,0) {\tiny $\bullet$};
\node[inner sep=1pt,circle] (11) at (1,1) {\tiny $\bullet$};
\node[inner sep=1pt,circle] (02) at (0,2) {\tiny $\bullet$};
\node[inner sep=1pt,circle] (13) at (1,3) {\tiny $\bullet$};

\draw[<-,green!50!black,dotted,thick] (00) -- (02); 
\draw[<-,green!50!black,dotted,thick] (11) -- (13); 
\draw[<-,red,dashed,thick] (00) -- (11);
\draw[<-,red,dashed,thick] (02) -- (13);

\draw[<-,blue,thick] (00) -- (6,0);
\draw[<-,blue,thick] (13) -- (7,3);
\draw[<-,blue,thick] (11) -- (7,1);
\draw[<-,blue,thick] (02) -- (6,2);

\end{scope}

\begin{scope}[yshift=-4.5cm,xshift=5cm]

\node[inner sep=1pt,circle] (00) at (0,0) {\tiny $\bullet$};
\node[inner sep=1pt,circle] (11) at (1,1) {\tiny $\bullet$};
\node[inner sep=1pt,circle] (02) at (0,2) {\tiny $\bullet$};
\node[inner sep=1pt,circle] (13) at (1,3) {\tiny $\bullet$};

\draw[<-,green!50!black,dotted,thick] (00) -- (02); 
\draw[<-,green!50!black,dotted,thick] (11) -- (13); 
\draw[<-,red,dashed,thick] (00) -- (11);
\draw[<-,red,dashed,thick] (02) -- (13);

\end{scope}


\begin{scope}[xshift=10cm]

\node[inner sep=1pt,circle] (00) at (0,0) {\tiny $ $};
\node[inner sep=1pt,circle] (11) at (1,1) {\tiny $ $};
\node[inner sep=1pt,circle] (02) at (0,2) {\tiny $ $};
\node[inner sep=1pt,circle] (13) at (1,3) {\tiny $ $};

\draw[-,dotted] (00) -- (02); 
\draw[-,dotted] (11) -- (13); 
\draw[-,dotted] (00) -- (11);
\draw[-,dotted] (02) -- (13);

\draw[-,dotted] (00) -- (6,0);
\draw[-,thick] (13) -- (7,3);
\draw[-,dotted] (11) -- (7,1);
\draw[-,dotted] (02) -- (6,2);

\draw[-,dotted] (00) -- (-1,-4.5);
\draw[-,dotted] (11) -- (0,-3.5);
\draw[-,dotted] (02) -- (-1,-2.5);
\draw[-,thick] (13) -- (0,-1.5);

\begin{scope}[xshift=6cm]

\node[inner sep=1pt,circle] (00) at (0,0) {\tiny $ $};
\node[inner sep=1pt,circle] (11) at (1,1) {\large $\square$};
\node[inner sep=1pt,circle] (02) at (0,2) {\tiny $ $};
\node[inner sep=1pt,circle] (13) at (1,3) {\tiny $ $};

\draw[-,dotted] (00) -- (02); 
\draw[-,thick] (11) -- (13); 
\draw[-,thick] (00) -- (11);
\draw[-,dotted] (02) -- (13);

\draw[-,thick] (00) -- (-1,-4.5);
\draw[-,thick] (11) -- (0,-3.5);
\draw[-,dotted] (02) -- (-1,-2.5);
\draw[-,dotted] (13) -- (0,-1.5);

\end{scope}

\begin{scope}[yshift=-4.5cm,xshift=-1cm]

\node[inner sep=1pt,circle] (00) at (0,0) {\tiny $ $};
\node[inner sep=1pt,circle] (11) at (1,1) {\large $\square$};
\node[inner sep=1pt,circle] (02) at (0,2) {\tiny $ $};
\node[inner sep=1pt,circle] (13) at (1,3) {\Large $\bullet$};

\draw[-,dotted] (00) -- (02); 
\draw[-,thick] (11) -- (13); 
\draw[-,thick] (00) -- (11);
\draw[-,dotted] (02) -- (13);

\draw[-,thick] (00) -- (6,0);
\draw[-,thick] (13) -- (7,3);
\draw[-,thick] (11) -- (7,1);
\draw[-,dotted] (02) -- (6,2);

\end{scope}

\begin{scope}[yshift=-4.5cm,xshift=5cm]

\node[inner sep=1pt,circle] (00) at (0,0) {\Large$\bullet$};
\node[inner sep=1pt,circle] (11) at (1,1) {\Large$\bullet$};
\node[inner sep=1pt,circle] (02) at (0,2) {\tiny $ $};
\node[inner sep=1pt,circle] (13) at (1,3) {\large $\square$};

\draw[-,thick] (00) -- (02); 
\draw[-,thick] (11) -- (13); 
\draw[-,dotted] (00) -- (11);
\draw[-,thick] (02) -- (13);

\end{scope}

\end{scope}

\end{tikzpicture}
\end{array}
\end{equation}
\end{example}

\noindent
\textbf{Fact 2.} Let $\Lambda$ be a rank-$1$ tree. Then $\Lambda$ does not admit a fixed point free automorphism of order three.


\vspace{3mm}

\noindent
The same is not true of higher-rank trees.
\vspace{3mm}

\noindent
Using the techniques developed in section~\ref{sec:glue-coloured} one may glue three locally convex, planar rank-$2$ trees together to produce a new locally convex, planar rank-$2$ tree  which has an automorphism of order three which admits no fixed points. Indeed, the reader will see that the construction given easily generalises.

\begin{example} \label{ex:triangle}
Consider the locally convex, planar rank-$2$ tree $\Lambda$ shown below to the left, it is formed from gluing six copies of the rank-$2$ graph $\Omega_{2,(1,1)}$ together together with the obvious $8$ squares $\mathcal{C}$. Note that it is singly connected as all four paths from $v_1$ to the central vertex ${\small \bullet_a}$, are equivalent. If we try to label the vertices of $E$ as if it were a member of $\mathcal{L}_{\mathfrak{P}}$ as shown in section~\ref{sec:Plabel}, then $v_1,v_2$ would be the only vertices and $\bullet_a$ would be the only vertex which corresponds to an arc, so $|A|=1$. Then by Theorem~\ref{thm:polygiveshrg} there would be $2 |A| = 2$ squares, but we have $| \mathcal{C}| =8$ here. Hence $\Lambda$ is not a member of $\mathcal{L}_{\mathfrak{P}}$ even though $|E^1|-2|E^0|+4=0$ (cf.\ Proposition~\ref{prp:Lambek_Euler}).
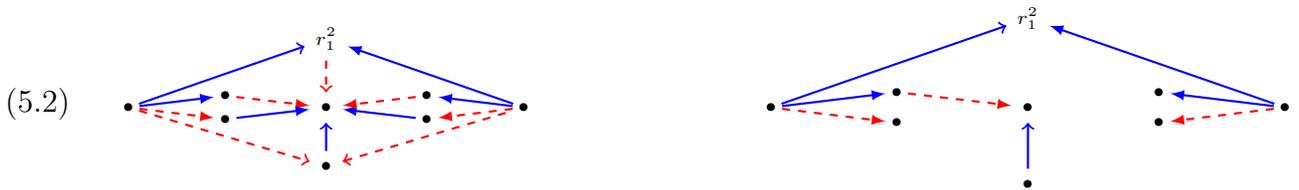
\begin{equation} \label{eq:diamond} 
\begin{array}{c}
\begin{tikzpicture}[scale=1.4,decoration={markings, mark=at position 0.6 with {\arrow{>}}}]
    
\node[inner sep=1pt, circle] (C) at (0,0) {\Tiny $\bullet$};
\node[inner sep=1pt, circle] (B) at (4,0) {\Tiny $\bullet$};
\node at (-0.2,0) {\Tiny $v_1$};
\node at (4.2,0) {\Tiny $v_2$};

\node (1c) at (2,0.7) {\Tiny  $r_1$};
\node (4) at (2,-0.6) { \Tiny $\bullet$};    

\draw[color=white] (B)-- node (R) [color=black,pos =0.5,inner sep=1pt] {\Tiny ${\hphantom{-}}\bullet_a$}(C); 
\draw[<-,thick,blue] (R) -- (4);
\draw[->,thick,dashed,red] (C) --  (4); 
\draw[<-,thick,dashed,red] (4) --  (B); 

\draw[->,thick,blue] (C) -- (1c); 
\draw[->,thick,dashed,red] (1c) --  (R); 
\draw[-latex,thick,blue] (B) -- (1c);

\node[inner sep=1pt, circle] (CR) at (0.98,0.12) {\Tiny $\bullet$};
\node[inner sep=1pt, circle] (CR1) at (0.98,-0.12) {\Tiny $\bullet$};

\node[inner sep=1pt, circle] (BR) at (3.02,-0.12) {\Tiny $\bullet$};
\node[inner sep=1pt, circle] (BR1) at (3.02,0.12) {\Tiny $\bullet$};

\draw[-latex,thick,blue] (C) -- (CR);
\draw[-latex,thick,red,dashed] (C) -- (CR1);
\draw[-latex,thick,red,dashed] (CR) -- (R);
\draw[-latex,thick,blue] (CR1) -- (R);
\draw[-latex,thick,dashed,red] (B) -- (BR);
\draw[-latex,thick,blue] (B) -- (BR1);
\draw[-latex,thick,blue] (BR) -- (R);
\draw[-latex,thick,dashed,red] (BR1) -- (R);

\begin{scope}[scale=1.2,xshift=4.6cm]
    
\node[inner sep=1pt, circle] (C) at (0,0) {\Tiny $\bullet$};
\node[inner sep=1pt, circle] (B) at (4,0) {\Tiny $\bullet$};
\node at (-0.2,0) {\Tiny $v_1$};
\node at (4.2,0) {\Tiny $v_2$};

\node (1c) at (2,0.7) {\Tiny  $r_1$};
\node (4) at (2,-0.6) { \Tiny $\bullet$};    

\draw[color=white] (B)-- node (R) [color=black,pos =0.5,inner sep=1pt] {\Tiny ${\hphantom{-}} \bullet_a$}(C); 
\draw[<-,thick,blue] (R) -- (4);

\draw[->,thick,blue] (C) -- (1c); 

\draw[-latex,thick,blue] (B) -- (1c);

\node[inner sep=1pt, circle] (CR) at (0.98,0.12) {\Tiny $\bullet$};
\node[inner sep=1pt, circle] (CR1) at (0.98,-0.12) {\Tiny $\bullet$};

\node[inner sep=1pt, circle] (BR) at (3.02,-0.12) {\Tiny $\bullet$};
\node[inner sep=1pt, circle] (BR1) at (3.02,0.12) {\Tiny $\bullet$};

\draw[-latex,thick,blue] (C) -- (CR);
\draw[-latex,thick,red,dashed] (C) -- (CR1);
\draw[-latex,thick,red,dashed] (CR) -- (R);

\draw[-latex,thick,red,dashed] (B) -- (BR);
\draw[-latex,thick,blue] (B) -- (BR1);

\end{scope}

\end{tikzpicture}
\end{array}
\end{equation}

\noindent
It has maximal spanning tree as shown above to the right and from there it is straightforward to show that it is a tree.
$\Lambda$ has sources at $v_1,v_2$ and so by the results of section~\ref{sec:glue-coloured} we may glue $\Lambda$ to itself to form a triangular $2$-coloured graph $\Delta$ shown below on the left. Either by inspection or repeated use of Theorem~\ref{thm:glue} we see that the coloured graph $\Delta$ shown below is a locally convex, planar rank-$2$ graph with spanning tree shown on the right.
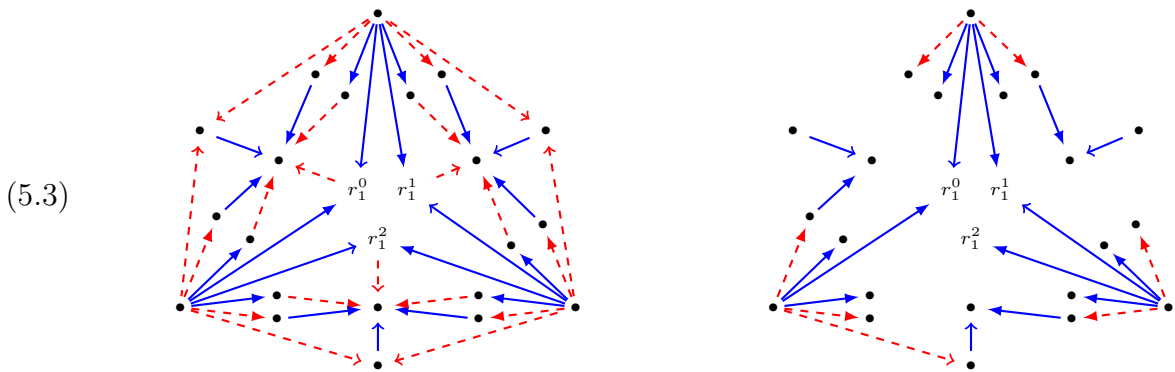
\begin{equation} \label{eq:daisy_chain}
\begin{array}{l}
\begin{tikzpicture}[scale=1.3,decoration={markings, mark=at position 0.6 with {\arrow{>}}}]
    
\node[inner sep=1pt, circle] (A) at (2,3) {\Tiny $\bullet$};
\node[inner sep=1pt, circle] (C) at (0,0) {\Tiny $\bullet$};
\node[inner sep=1pt, circle] (B) at (4,0) {\Tiny $\bullet$};
    
\node (1a) at (1.8,1.2) {\Tiny  $r_1$};
\node (1b) at (2.3,1.2) {\Tiny  $r_2$};
\node (1c) at (2,0.7) {\Tiny  $r_3$};

\node (2) at (0.2,1.8) { \Tiny $\bullet$};  

\node (3) at (3.7,1.8) { \Tiny $\bullet$};  

\node (4) at (2,-0.6) { \Tiny $\bullet$};    

\draw[color=white] (C)-- node (P) [color=black,pos =0.5] {\Tiny $\bullet$} (A);

\draw[color=white] (B)--node (Q) [color=black,pos =0.5,] {\Tiny $\bullet$} (A) ;

\draw[color=white] (B)-- node (R) [color=black,pos =0.5] {\Tiny $\bullet$}(C);


\draw[-latex,thick,dashed,red] (A) --  (2); %

\draw[-latex,thick,blue] (2) -- (P);

\draw[latex-,thick,dashed,red] (P) -- (1a); %

\draw[latex-,thick,blue] (1a) --  (A); %

\draw[latex-,thick,dashed,red] (2) -- (C);%

\draw[-latex,thick,blue] (C) -- (1c); %
 
\draw[-latex,thick,dashed,red] (1c) --  (R); %

\draw[latex-,thick,blue] (R) -- (4);%
 
\draw[-latex,thick,dashed,red] (C) --  (4); %

\draw[latex-,thick,dashed,red] (4) --  (B); %

\draw[-latex,thick,blue] (B) -- (1b); %

\draw[latex-,thick,dashed,red] (Q) --   (1b); 

\draw[-latex,thick,blue] (3) --  (Q);

\draw[latex-,thick,dashed,red] (3) --  (A); %

\draw[latex-,thick,dashed,red] (3) --  (B); %


\node[inner sep=1pt, circle] (CP) at (0.37,0.92) {\Tiny $\bullet$}; %

\node[inner sep=1pt, circle] (CP1) at (0.71,0.69) {\Tiny $\bullet$}; %

\node[inner sep=1pt, circle] (AP) at (1.67,2.16) {\Tiny $\bullet$};
\node[inner sep=1pt, circle] (AP1) at (1.37,2.37) {\Tiny $\bullet$};

\node[inner sep=1pt, circle] (AQ) at (2.65,2.37) {\Tiny $\bullet$};
\node[inner sep=1pt, circle] (AQ1) at (2.33,2.16) {\Tiny $\bullet$};

\node[inner sep=1pt, circle] (BQ) at (3.35,0.63) {\Tiny $\bullet$};
\node[inner sep=1pt, circle] (BQ1) at (3.67,0.84) {\Tiny $\bullet$};

\node[inner sep=1pt, circle] (CR) at (0.98,0.12) {\Tiny $\bullet$};
\node[inner sep=1pt, circle] (CR1) at (0.98,-0.12) {\Tiny $\bullet$};

\node[inner sep=1pt, circle] (BR) at (3.02,-0.12) {\Tiny $\bullet$};
\node[inner sep=1pt, circle] (BR1) at (3.02,0.12) {\Tiny $\bullet$};


\draw[-latex,thick,dashed,red] (C) -- (CP);
\draw[-latex,thick,blue] (C) -- (CP1);
\draw[-latex,thick,blue] (CP) -- (P);
\draw[-latex,thick,dashed,red] (CP1) -- (P);
\draw[-latex,thick,blue] (A) -- (AP);
\draw[-latex,thick,red,dashed] (A) -- (AP1);
\draw[-latex,thick,red,dashed] (AP) -- (P);
\draw[-latex,thick,blue] (AP1) -- (P);

\draw[-latex,thick,blue] (B) -- (BQ);
\draw[-latex,thick,dashed,red] (B) -- (BQ1);
\draw[-latex,thick,red,dashed] (BQ) -- (Q);
\draw[-latex,thick,blue] (BQ1) -- (Q);
\draw[-latex,thick,dashed,red] (A) -- (AQ);
\draw[-latex,thick,blue] (A) -- (AQ1);
\draw[-latex,thick,blue] (AQ) -- (Q);
\draw[-latex,thick,dashed,red] (AQ1) -- (Q);

\draw[-latex,thick,blue] (C) -- (CR);
\draw[-latex,thick,red,dashed] (C) -- (CR1);
\draw[-latex,thick,red,dashed] (CR) -- (R);
\draw[-latex,thick,blue] (CR1) -- (R);
\draw[-latex,thick,dashed,red] (B) -- (BR);
\draw[-latex,thick,blue] (B) -- (BR1);
\draw[-latex,thick,blue] (BR) -- (R);
\draw[-latex,thick,dashed,red] (BR1) -- (R);


\draw[-latex,thick,blue] (C) -- (1a);
\draw[-latex,thick,blue] (A) -- (1b);
\draw[-latex,thick,blue] (B) -- (1c);

\begin{scope}[xshift=6cm]
\node[inner sep=1pt, circle] (A) at (2,3) {\Tiny $\bullet$};
\node[inner sep=1pt, circle] (C) at (0,0) {\Tiny $\bullet$};
\node[inner sep=1pt, circle] (B) at (4,0) {\Tiny $\bullet$};
    
\node (1a) at (1.8,1.2) {\Tiny  $r_1$};
\node (1b) at (2.3,1.2) {\Tiny  $r_2$};
\node (1c) at (2,0.7) {\Tiny  $r_3$};

\node (2) at (0.2,1.8) { \Tiny $\bullet$};  

\node (3) at (3.7,1.8) { \Tiny $\bullet$};  

\node (4) at (2,-0.6) { \Tiny $\bullet$};

\draw[color=white] (C)-- node (P) [color=black,pos =0.5] {\Tiny $\bullet$} (A);

\draw[color=white] (B)--node (Q) [color=black,pos =0.5,] {\Tiny $\bullet$} (A) ;

\draw[color=white] (B)-- node (R) [color=black,pos =0.5] {\Tiny $\bullet$}(C);


\draw[-latex,thick,blue] (2) -- (P);%

\draw[latex-,thick,blue] (1a) --  (A); %

\draw[latex-,thick,blue] (R) -- (4);%

\draw[-latex,thick,dashed,red] (C) --  (4); %

\draw[-latex,thick,blue] (B) -- (1b); %

\draw[-latex,thick,blue] (3) --  (Q);%


\node[inner sep=1pt, circle] (CP) at (0.37,0.92) {\Tiny $\bullet$}; 

\node[inner sep=1pt, circle] (CP1) at (0.71,0.69) {\Tiny $\bullet$}; %

\node[inner sep=1pt, circle] (AP) at (1.67,2.16) {\Tiny $\bullet$};
\node[inner sep=1pt, circle] (AP1) at (1.37,2.37) {\Tiny $\bullet$};

\node[inner sep=1pt, circle] (AQ) at (2.65,2.37) {\Tiny $\bullet$};
\node[inner sep=1pt, circle] (AQ1) at (2.33,2.16) {\Tiny $\bullet$};

\node[inner sep=1pt, circle] (BQ) at (3.35,0.63) {\Tiny $\bullet$};
\node[inner sep=1pt, circle] (BQ1) at (3.67,0.84) {\Tiny $\bullet$};

\node[inner sep=1pt, circle] (CR) at (0.98,0.12) {\Tiny $\bullet$};
\node[inner sep=1pt, circle] (CR1) at (0.98,-0.12) {\Tiny $\bullet$};

\node[inner sep=1pt, circle] (BR) at (3.02,-0.12) {\Tiny $\bullet$};
\node[inner sep=1pt, circle] (BR1) at (3.02,0.12) {\Tiny $\bullet$};


\draw[-latex,thick,dashed,red] (C) -- (CP);
\draw[-latex,thick,blue] (C) -- (CP1);
\draw[-latex,thick,blue] (CP) -- (P);

\draw[-latex,thick,blue] (A) -- (AP);
\draw[-latex,thick,red,dashed] (A) -- (AP1);

\draw[-latex,thick,blue] (B) -- (BQ);
\draw[-latex,thick,dashed,red] (B) -- (BQ1);

\draw[-latex,thick,dashed,red] (A) -- (AQ);
\draw[-latex,thick,blue] (A) -- (AQ1);
\draw[-latex,thick,blue] (AQ) -- (Q);

\draw[-latex,thick,blue] (C) -- (CR);
\draw[-latex,thick,red,dashed] (C) -- (CR1);

\draw[-latex,thick,dashed,red] (B) -- (BR);
\draw[-latex,thick,blue] (B) -- (BR1);
\draw[-latex,thick,blue] (BR) -- (R);


\draw[-latex,thick,blue] (C) -- (1a);
\draw[-latex,thick,blue] (A) -- (1b);
\draw[-latex,thick,blue] (B) -- (1c);
\end{scope}

\end{tikzpicture}
\end{array}
\end{equation}

\noindent
By construction the graph $\Delta$ shown above admits a fixed point free action of $\mathbb{Z} / 3 \mathbb{Z} $ by rotating it by $\frac{2\pi}{3}$ anti-clockwise about the centre of the equilateral triangle formed by $r_1,r_2,r_3$ shown above. It remains to find out if it is a rank-$2$ tree.

To compute its fundamental group we find a maximal spanning tree as shown above to the right. Note that we cannot use the same spanning tree for $\Lambda$ in all three places, as it would create a nontrivial blue cycle on the inside of $\Delta$.
A short calculation using the maximal spanning tree shown in \eqref{eq:daisy_chain} above right shows that the fundamental group is trivial and so the graph is a rank-$2$ tree, and our task is complete.
\end{example}


\textbf{Fact 3.} Let $\Lambda$ be a rank-$1$ tree with an odd number of vertices. The $\Lambda$ admits a fixed point free automorphism.

\vspace{3mm}
The same is not true of higher-rank trees.

\begin{example} \label{ex:tree-glue}
The following connected, singly connected locally convex, planar rank-$2$ tree built out of gluing copies of $\Omega_{2,(1,1)}$ together  has $15$ vertices, but admits no automorphisms. It also embeds in its fundamental groupoid.
\begin{equation} \label{eq:tree-glue}
\begin{array}{l}
\begin{tikzpicture}[scale=1.3]

\node[inner sep=2.5pt] at (0,1) {\tiny $\bullet$};
\node[inner sep=2.5pt] at (0,2) {\tiny $\bullet$};


\node[inner sep=2.5pt] at (1,1) {\tiny $\bullet$};
\node[inner sep=2.5pt] at (1,2) {\tiny $\bullet$};


\node[inner sep=2.5pt] at (2,1) {\tiny $\bullet$};
\node[inner sep=2.5pt] at (2,2) {\tiny $\bullet$};


\node[inner sep=2.5pt] at (3,0) {\tiny $\bullet$};
\node[inner sep=2.5pt] at (3,1) {\tiny $\bullet$};
\node[inner sep=2.5pt] at (3,2) {\tiny $\bullet$};


\node[inner sep=2.5pt] at (4,0) {\tiny $\bullet$};
\node[inner sep=2.5pt] at (4,1) {\tiny $\bullet$};
\node[inner sep=2.5pt] at (4,2) {\tiny $\bullet$};


\node[inner sep=2.5pt] at (5,0) {\tiny $\bullet$};
\node[inner sep=2.5pt] at (5,1) {\tiny $\bullet$};
\node[inner sep=2.5pt] at (5,2) {\tiny $\bullet$};

\draw[blue, thick,midarrow] (4,0)--(3,0);
\draw[blue,thick,midarrow] (5,0)--(4,0);

\draw[blue, thick,midarrow] (1,1)--(0,1);
\draw[blue, thick,midarrow] (2,1)--(1,1);
\draw[blue, thick,midarrow] (3,1)--(2,1);
\draw[blue, thick,midarrow] (4,1)--(3,1);
\draw[blue,thick,midarrow] (5,1)--(4,1);

\draw[blue, thick,midarrow] (1,2)--(0,2);
\draw[blue, thick,midarrow] (2,2)--(1,2);
\draw[blue, thick,midarrow] (3,2)--(2,2);
\draw[blue, thick,midarrow] (4,2)--(3,2);
\draw[blue,thick,midarrow] (5,2)--(4,2);


\draw[red, dashed, thick,midarrow] (0,2)--(0,1);

\draw[red, dashed, thick,midarrow] (1,2)--(1,1);

\draw[red, dashed, thick,midarrow] (2,2)--(2,1);

\draw[red, dashed, thick,midarrow] (3,1)--(3,0);
\draw[red, dashed, thick,midarrow] (3,2)--(3,1);

\draw[red, dashed, thick,midarrow] (4,1)--(4,0);
\draw[red, dashed, thick,midarrow] (4,2)--(4,1);

\draw[red, dashed, thick,midarrow] (5,1)--(5,0);
\draw[red, dashed, thick,midarrow] (5,2)--(5,1);

%
%
%

\end{tikzpicture}
\end{array}
\end{equation}
\end{example}

\section{Tailpiece} \label{sec:tailpiece}

Many examples of higher-rank trees we have covered in the above text have sources and/or sinks on the outside of the main body of their skeleton. See \eqref{eq:tree-glue}, \eqref{eq:daisy_chain}, \eqref{eq:diamond}, \eqref{eq:q4}, \eqref{eq:c1-coloured}. Some of these graphs may be glued together at sources as we saw with \eqref{eq:daisy_chain}, \eqref{eq:diamond}; here we glued them together in a circuit using the two sources on the outside of \eqref{eq:daisy_chain}. Some of these graphs, typically members of $\mathcal{L}_{\mathfrak{P}}$ which have sinks on the outside as with \eqref{eq:c1-coloured}. Again, we can glue them together in a circuit, however we can also glue them in place of the edges in a 2-regular tree $T_2$ as shown below.
\[
\begin{tikzpicture}

\begin{scope}[xscale=0.03cm,rotate=90,yscale=0.03cm]

\node[inner sep=1pt,circle] (p0) at (-1,0) {\Tiny  $\bullet$};
\node[inner sep=1pt,circle] (p1) at (1,0) {\Tiny  $\bullet$};

\node[inner sep=1pt,circle] (f0) at (0,0) {\Tiny  $\bullet$};

\node[inner sep=1pt,circle] (lf1) at (-2,0) {\Tiny  $\bullet$};
\node[inner sep=1pt,circle] (rf1) at (2,0) {\Tiny  $\bullet$};

\node[inner sep=1pt,circle] (a1) at (0,-1.5) {\Tiny  $\bullet$};
\node[inner sep=1pt,circle] (a0) at (0,1.5) {\Tiny  $\bullet$};

\draw[thick,blue,-latex] (p0) -- (f0);
\draw[dashed,red,thick,-latex] (p0) -- (lf1);
\draw[thick,blue,-latex] (p1) -- (f0);
\draw[red,dashed,thick,-latex] (p1) -- (rf1);


\draw[thick,blue,-latex] (lf1) -- (a0);
\draw[red,dashed,thick,-latex] (f0) -- (a0);
\draw[blue,thick,-latex] (rf1) -- (a1);
\draw[red,dashed,thick,-latex] (f0) -- (a1);


\draw[thick,blue,-latex] (rf1) -- (a0);
\draw[thick,blue,-latex] (lf1) -- (a1);

\end{scope}

\begin{scope}[xscale=0.03cm,rotate=90,yscale=0.03cm,yshift=-3cm]

\node[inner sep=1pt,circle] (p0) at (-1,0) {\Tiny  $\bullet$};
\node[inner sep=1pt,circle] (p1) at (1,0) {\Tiny  $\bullet$};

\node[inner sep=1pt,circle] (f0) at (0,0) {\Tiny  $\bullet$};

\node[inner sep=1pt,circle] (lf1) at (-2,0) {\Tiny  $\bullet$};
\node[inner sep=1pt,circle] (rf1) at (2,0) {\Tiny  $\bullet$};

\node[inner sep=1pt,circle] (a1) at (0,-1.5) {\Tiny  $\bullet$};
\node[inner sep=1pt,circle] (a0) at (0,1.5) {\Tiny  $\bullet$};

\node[inner sep=1pt,circle] at (0,-2) { $\ldots$};

\draw[thick,blue,-latex] (p0) -- (f0);
\draw[dashed,red,thick,-latex] (p0) -- (lf1);
\draw[thick,blue,-latex] (p1) -- (f0);
\draw[red,dashed,thick,-latex] (p1) -- (rf1);


\draw[thick,blue,-latex] (lf1) -- (a0);
\draw[red,dashed,thick,-latex] (f0) -- (a0);
\draw[blue,thick,-latex] (rf1) -- (a1);
\draw[red,dashed,thick,-latex] (f0) -- (a1);


\draw[thick,blue,-latex] (rf1) -- (a0);
\draw[thick,blue,-latex] (lf1) -- (a1);

\end{scope}

\begin{scope}[xscale=0.03cm,rotate=90,yscale=0.03cm,yshift=3cm]

\node[inner sep=1pt,circle] (p0) at (-1,0) {\Tiny  $\bullet$};
\node[inner sep=1pt,circle] (p1) at (1,0) {\Tiny  $\bullet$};

\node[inner sep=1pt,circle] (f0) at (0,0) {\Tiny  $\bullet$};

\node[inner sep=1pt,circle] (lf1) at (-2,0) {\Tiny  $\bullet$};
\node[inner sep=1pt,circle] (rf1) at (2,0) {\Tiny  $\bullet$};

\node[inner sep=1pt,circle] (a1) at (0,-1.5) {\Tiny  $\bullet$};
\node[inner sep=1pt,circle] (a0) at (0,1.5) {\Tiny  $\bullet$};

\node[inner sep=1pt,circle] at (0,2) { $\ldots$}; 

\draw[thick,blue,-latex] (p0) -- (f0);
\draw[dashed,red,thick,-latex] (p0) -- (lf1);
\draw[thick,blue,-latex] (p1) -- (f0);
\draw[red,dashed,thick,-latex] (p1) -- (rf1);


\draw[thick,blue,-latex] (lf1) -- (a0);
\draw[red,dashed,thick,-latex] (f0) -- (a0);
\draw[blue,thick,-latex] (rf1) -- (a1);
\draw[red,dashed,thick,-latex] (f0) -- (a1);


\draw[thick,blue,-latex] (rf1) -- (a0);
\draw[thick,blue,-latex] (lf1) -- (a1);

\end{scope}

\begin{scope}[xscale=0.008cm,rotate=180,yscale=0.03cm,xshift=4.7cm,yshift=1.5cm]

\node[inner sep=1pt,circle] (p0) at (-1,0) {\Tiny  $\bullet$};
\node[inner sep=1pt,circle] (p1) at (1,0) {\Tiny  $\bullet$};

\node[inner sep=1pt,circle] (f0) at (0,0) {\Tiny  $\bullet$};

\node[inner sep=1pt,circle] (lf1) at (-2,0) {\Tiny  $\bullet$};
\node[inner sep=1pt,circle] (rf1) at (2,0) {\Tiny  $\bullet$};

\node[inner sep=1pt,circle] (a1) at (0,-1.5) {\Tiny  $\bullet$};
\node[inner sep=1pt,circle] (a0) at (0,1.5) {\Tiny  $\bullet$};

\node[inner sep=1pt,circle] at (-1.3,1.7) { $\ddots$};
\node[inner sep=1pt,circle] at (1,1.7) { $\iddots$};

\draw[thick,blue,-latex] (p0) -- (f0);
\draw[dashed,red,thick,-latex] (p0) -- (lf1);
\draw[thick,blue,-latex] (p1) -- (f0);
\draw[red,dashed,thick,-latex] (p1) -- (rf1);


\draw[thick,blue,-latex] (lf1) -- (a0);
\draw[red,dashed,thick,-latex] (f0) -- (a0);
\draw[blue,thick,-latex] (rf1) -- (a1);
\draw[red,dashed,thick,-latex] (f0) -- (a1);


\draw[thick,blue,-latex] (rf1) -- (a0);
\draw[thick,blue,-latex] (lf1) -- (a1);

\end{scope}

\begin{scope}[xscale=0.008cm,rotate=180,yscale=0.03cm,xshift=-4.7cm,yshift=1.5cm]

\node[inner sep=1pt,circle] (p0) at (-1,0) {\Tiny  $\bullet$};
\node[inner sep=1pt,circle] (p1) at (1,0) {\Tiny  $\bullet$};

\node[inner sep=1pt,circle] (f0) at (0,0) {\Tiny  $\bullet$};

\node[inner sep=1pt,circle] (lf1) at (-2,0) {\Tiny  $\bullet$};
\node[inner sep=1pt,circle] (rf1) at (2,0) {\Tiny  $\bullet$};

\node[inner sep=1pt,circle] (a1) at (0,-1.5) {\Tiny  $\bullet$};
\node[inner sep=1pt,circle] (a0) at (0,1.5) {\Tiny  $\bullet$};

\node[inner sep=1pt,circle] at (-1.3,1.7) { $\ddots$};
\node[inner sep=1pt,circle] at (1,1.7) { $\iddots$};

\draw[thick,blue,-latex] (p0) -- (f0);
\draw[dashed,red,thick,-latex] (p0) -- (lf1);
\draw[thick,blue,-latex] (p1) -- (f0);
\draw[red,dashed,thick,-latex] (p1) -- (rf1);


\draw[thick,blue,-latex] (lf1) -- (a0);
\draw[red,dashed,thick,-latex] (f0) -- (a0);
\draw[blue,thick,-latex] (rf1) -- (a1);
\draw[red,dashed,thick,-latex] (f0) -- (a1);


\draw[thick,blue,-latex] (rf1) -- (a0);
\draw[thick,blue,-latex] (lf1) -- (a1);

\end{scope}

\end{tikzpicture}
\]

\noindent
The automorphism group of the above graph contains the automorphism group of $T_2$, plus
an internal symmetry $\mathbb{Z} / 2 \mathbb{Z}$ for each of the diamonds substituted for each edge.

\begin{dfn}
Let $\alpha, \beta \in \Lambda^{\le m}$,  set
\[
\Sigma_{\alpha,\beta} = \{ g \in \Gamma : g ( \alpha ) = g ( \beta ) \} .
\]
\end{dfn}

\noindent
Let $\gamma,\delta \in \Lambda^{\le n}$ and consider $\Sigma_{\alpha,\beta} \cap \Sigma_{\gamma,\delta}$. If $n \le m$ then is nonempty unless 
$\gamma = \alpha \alpha'$ and $\delta = \beta \beta'$ in which case $\Sigma={\alpha,\beta} \cap \Sigma_{\alpha\alpha',\beta\beta'} = \Sigma_{\alpha,\beta}$. Conversely, in the case $m \le n$ we have $\alpha = \gamma \gamma'$ and $\beta = \delta  \delta'$ we have $\Sigma_{\gamma \gamma',\beta \beta'} \cap \Sigma_{\gamma ,\delta} = \Sigma_{\gamma,\delta}$. Suppose $m,n$ are incomparable and $g \in \Sigma_{\alpha,\beta} \cap \Sigma_{\gamma,\delta}$. Let $t = m \wedge n$ then by the factorisation property we have
\[
g ( \alpha (0,t) ) = g ( \beta (0,t) ) \text{ from } \Sigma_{\alpha,\beta} \text{ and } 
g ( \gamma (0,t) ) = g ( \delta (0,t) ) \text{ from } \Sigma_{\gamma,\delta} .
\]

\noindent
Hence $g \in \Sigma_{\alpha(0,t),\beta(0,t)}$ and $g \in \Sigma_{\gamma(0,t),\delta(0,t)}$.
Arguing as before, we can only have a nonempty intersection if $\gamma (0,t)  = \mu \nu$ and $\alpha (0,t) = \mu \nu'$, similarly $\delta (0,t)  = \kappa \nu''$ and $\beta (0,t) = \kappa \nu'''$, so  $\Sigma_{\alpha,\beta} \cap \Sigma_{\gamma,\delta} = \Sigma_{\mu,\kappa}$. The above calculations show that $\{ \Sigma_{\alpha,\beta} : \alpha ,\beta \in \Lambda^{\le m}  \}$ form a basis for a topology on $\Gamma$.

\begin{thm}
Let $\Lambda$ be a $k$-graph and $\Gamma = \Aut ( \Lambda )$ be the group of automorphisms of $\Lambda$. Then $\Gamma$ with the pointwise convergence topology is a totally disconnected, locally compact group.
\end{thm}

\begin{proof}[{\bf Proof} (Following {\cite[\S 1.3]{We}})]
Take basic open set $\Sigma_{\alpha, \beta}$ and consider the inversion map $i : \Gamma \to \Gamma$ given by $i(g)=g^{-1}$. The preimage of $\Sigma_{\alpha,\beta}$ is then $\Sigma_{\beta,\alpha}$ which is an open set. Let $m : \Gamma \times \Gamma \to \Gamma$ be the multiplication map. Fix $(g , h ) \in m^{-1} ( \Sigma_{\alpha,\beta} )$. Since $g$ is bijective there must be $\gamma \in \Lambda^{\le m}$ such that $g ( \alpha ) = \gamma$. Since $hg \in \Sigma_{\alpha,\beta}$ we must have $h ( \gamma ) = \beta$. The open set $\Sigma_{\gamma, \beta} \times \Gamma_{\alpha,\gamma}$ is then an open set containing $(h,g)$ and is contained in $m^{-1} ( \Sigma_{\alpha,\beta} )$, which is therefore open.

Since $\Lambda$, by definition, is countable it follows that $\Aut ( \Lambda )$ is second countable. Now suppose $F \subseteq \Lambda^0$ be finite. Let $\Gamma_{(F)}$ be the pointwise stabiliser of $F$ in $\Gamma$. The set $\Gamma_{(F)}$ is a basic open set and
\[
\mathcal{F} = \{ \Gamma_{(F)} : F \subseteq \Lambda^0 \text{ with } |F|<\infty \} 
\]

\noindent
is a basis of the identity. The sets $\Gamma_{(F)}$ are subgroups with nonempty interior, so $\mathcal{F}$ is in fact a basis of clopen subsets. Since a basis for the topology on $\Aut ( \Lambda )$ is given by cosets of the elements of $\mathcal{F}$ it follows that $\Aut ( \Lambda )$ is totally disconnected.
\end{proof}


\begin{thebibliography}{xyzzz}

\bibitem{AH1} K.\ Appel and W.\ Haken, \emph{Every Planar Map is Four Colorable I, Discharging,} Illinois Journal of Mathematics, \textbf{21} (1977),  429--490.

\bibitem{AHK} K.\ Appel, W.\ Haken and J.\  Koch, \emph{Every Planar Map is Four Colorable II, Reducibility,} Illinois Journal of Mathematics, \textbf{21} 491--567.

\bibitem{AH2} K.\ Appel and W.\ Haken,  \emph{Solution of the Four Color Map Problem}, Scientific American, \emph{237} pp. 108--121.



\bibitem{Bo} B.\ Bollob\'as.
\newblock \emph{Modern Graph Theory}. Graduate Texts in Mathematics \textbf{184}, Springer 1998. 



\bibitem{bkps} N.\ Brownlowe, A.\ Kumjian, D.\ Pask and A.\ Sims. \emph{Embedding a higher-rank graph in its fundamental groupoid}. \href{ArXiv.org/pdf/2403.01337pdf}{ArXiv.org/pdf/2403.01337pdf}.

\bibitem{B} G.C.\ Bush.
\newblock \emph{The embedability of a semigroup -- conditions common to Malcev and Lambek,}
\newblock Trans.\ Amer.\ Math.\ Soc., \textbf{157} (1971), 437--448.

\bibitem{CFaH} L.\ Orloff Clark,  C.\ Flynn and A.\ an Huef, \emph{Kumjian-Pask algebras of locally convex higher rank graphs}, J.\ Algebra \textbf{399} (2014), 445--474.

\bibitem{D} R.\ Diestel, \emph{Graph Theory (Electronic Edition),} Graduate Texts in Mathematics Vol.\ \textbf{ 173},
Springer-Verlag,  Hiedelberg, New York (1997, 2000, 2005).

\bibitem{EH} M.\ Edwardes and D.\ Heath. \emph{A collection of cancellative, singly aligned, non-group embeddable monoids}, Semigroup Forum, (online) \href{doi.org/10.1007/s00233-025-10509-2}{doi.org/10.1007/s00233-025-10509-2}.

\bibitem{E} D. G.\ Evans, {\em On the $K$-theory of higher
 rank graph $C^*$-algebras,} New York J.\ Math., \textbf{14}, 2008, 1--31.


\bibitem{FGLP} C.\ Farsi, E.\ Gillaspy, N.\ Larsen and J. Packer. \emph{Generalized gauge actions on $k$-graph $C^*$-algebras: KMS states and Hausdorff structure,} Indiana Univ.\ Math.\ J.\ \textbf{70} (2021),  669–-709.


\bibitem{hrsw}
\newblock R.~Hazlewood, I.~Raeburn, A.~Sims, and S.B.G. Webster.
\newblock \emph{Remarks on some fundamental results about higher-rank graphs and their {$C^*$}-algebras,}
\newblock Proc.\ Edinb.\ Math.\ Soc., \textbf{56} (2013), 575--597.

\bibitem{Holl}
\newblock C.~Hollings.
\newblock \emph{Embedding semigroups in groups: not as easy as it might seem,}
\newblock Arch.\ Hist.\ Exact.\ Sci.\ (6) \textbf{68} (2014) 641--692.


\bibitem{Johnstone} P.T. Johnstone, \emph{On embedding categories in groupoids}, Math. Proc. Cambridge Philosoph. Soc. \textbf{145} (2008), 273--294.



\bibitem{KKQS} S.\ Kaliszewski, A.\ Kumjian, J.\ Quigg and A.\ Sims. \emph{Topological realizations and fundamental groups of higher-rank graphs}, Proc.\ Edinb.\ Math.\ Soc., \textbf{59} (2016), 143--168.



\bibitem{KPW} S.\ Kang, D.\ Pask  and S.B.G.\ Webster, \emph{Computing the fundamental group of a higher-rank graph}.  Proc.\ Endin.\ Math.\ Soc., \textbf{64} (2021), 650--661.

\bibitem{kpsw} A.\ Kumjian, D.\ Pask, A.\ Sims and M.\ Whittaker. \emph{Topological spaces associated to higher-rank graphs}. J.\ Comb.\ Thy.\ Ser.\ A, \textbf{143} (2016), 19--41.



\bibitem{KP1} A.~Kumjian and D.~Pask.
\newblock $C^*$-algebras of directed graphs and group actions.
\newblock Ergod.\ Th.\ \& Dyn.\ Sys., \textbf{19 }(1999), 1503--1519.

\bibitem{KP2} A.~Kumjian and D.~Pask.
\newblock \emph{Higher rank graph $C^*$-algebras,}
\newblock New York J.\ Math., {\bf 6} (2000), 1--20.

\bibitem{KPS1} A.~Kumjian, D.~Pask and A.~Sims. \emph{Generalised morphisms of $k$-graphs: $k$-morphs}, Trans.\ Amer.\ Math.\ Soc.,  {\bf 363} (2011), 2599--2626. 

\bibitem{KPS3}{A. Kumjian, D. Pask, and A. Sims}, \emph{Homology for higher-rank graphs and twisted $C^*$-algebras}. {\it J.\ Funct.\ Anal.}, {\bf 263} (2012), 1539--1574.

\bibitem{KPSW} A.\ Kumjian, D.\ Pask, A.\ Sims and M.\ Whittaker. \emph{Topological spaces associated to higher-rank graphs}. J.\ Comb.\ Thy.\ Ser.\ A, \textbf{143} (2016), 19--41.

\bibitem{Kur} A.\ Kuratowski. \emph{Sur le probl'{e}me des courbes gauches en topologie,} Fundamenta Mathematicae, \textbf{15} 271--283.


\bibitem{Lambek} J. Lambek, \emph{The immersibility of a semigroup into a group}, Canadian J.\ Math. \textbf{3} (1951), 34--43.

\bibitem{LPIS}
H. Larki: \emph{Purely infinite simple Kumjian-Pask algebras}, Forum Math. \textbf{30} (1) (2018), 253-268.

\bibitem{LV} M.\ Lawson and A.\ Vdovina. \emph{Higher dimensional generalisations of the Thompson groups}, Advances in Math.\ {\bf 369} (2020) 107--191.


\bibitem{Malcev} A.\ Malcev.
\newblock \emph{On the immersion of an algebraic ring into a field,}
\newblock Math. Ann. \textbf{113} (1937), 686–-691.


\bibitem{MSt} Math.Stackexchange. 
\newblock \href{https://math.stackexchange.com/questions/1350552/prove-that-the-graph-dual-to-eulerian-planar-graph-is-bipartite}{https://math.stackexchange.com/questions/1350552/prove-that-the-graph-dual-to-eulerian-planar-graph-is-bipartite}


\bibitem{PQR1} D.~Pask, I.~Raeburn and J.~Quigg.
\newblock \emph{Fundamental Groupoids of $k$-graphs,}
\newblock New York J.\ Math., {\bf 10} (2004), 195--207.

\bibitem{PQR2} D.~Pask, I.~Raeburn and J.~Quigg.
\newblock \emph{Coverings of $k$-graphs,}
\newblock J.\ Algebra, {\bf 289} (2005), 161--191.





\bibitem{Pre} R.\ Preusser. \emph{Simple modules for Kumjian-Pask algebras,} Algebr.\ Represent.\ Theor., \textbf{26}, (2023) 1271-–1293.



\bibitem{RSY1} I.\ Raeburn, A.\ Sims, Aidan and T.\ Yeend.
\emph{Higher-rank graphs and their {$C^*$}-algebras},
Proc.\ Edinb.\ Math.\ Soc., \textbf{46} (2003), 99--115.

\bibitem{RSY2} I.\ Raeburn,  A.\  Sims and T.\ Yeend.   
\emph{The {$C^*$}-algebras of finitely aligned higher-rank graphs}, J.\ Funct.\ Anal.,  \textbf{213} (2004), 206--240.



\bibitem{RobSteg1} G.\ Robertson and T.\ Steger,
\textit{Affine buildings, tiling systems and higher rank Cuntz-Krieger algebras}, J. reine angew.\ Math.,  \textbf{513} (1999), 115--144. 




\bibitem{Ros} R.\ Rosjanuardi. \emph{Complex Kumjian-Pask algebras,} Acta Math.\ Sin.\ (Engl. Ser.), \textbf{29} (2013), 2073--2078.



\bibitem{Schubert} H.~Schubert.
\newblock Categories,
\newblock Springer-Verlag, 1972.


\bibitem{Sp} J.\ Spielberg, \emph{Graph-based models for Kirchberg algebras,} J.\ Operator Theory \textbf{57} (2007), 347--374.

\bibitem{St} E.\ Steinitz. \emph{Polyeder und Raumeinteilungen,} Encyclop{\"a}die der mathematischen 
Wissenschaften mit Einschluss ihrer Anwendungen, vol.\ Band \textbf{3} (Geometries)  (1916), 1--139



\bibitem{Wikipedia2024} \emph{Tesseract.} Wikipedia, The Free Encyclopedia. Date retrieved: 16 July 2024 03:50 UTC. Page Version ID: 1232778451. \href{https://en.wikipedia.org/w/index.php title=Tesseract\&oldid=1232778451}{https://en.wikipedia.org/w/index.php title=Tesseract\&oldid=1232778451}.

\bibitem{We} P.\ Wesolek. \emph{An introduction to totally disconnected locally compact groups,}


\bibitem{W} G.\ Willis. \emph{The structure of totally disconnected locally compact groups,} Math.\ Annalen, \textbf{300} (1994), 341--363.

\bibitem{Y1}  D.\ Yang. \emph{Endomorphisms and modular theory of $2$-graph $C^*$-algebras,} Indiana Univ.\ Math.\ J., \textbf{59} (2010) 495--520. 

\bibitem{Y} D.\ Yang. \emph{The interplay between $k$-graphs and the Yang-Baxter equation,} J.\ Algebra \textbf{451} (2016), 494--525. 


\end{thebibliography}
\end{document}